\documentclass{jie-public}
\usepackage{cite}
\usepackage{hyperref}
\hypersetup{ 
	colorlinks=true, 
	citecolor=red}
\newtheorem{thm}{Theorem}[section]

\newtheorem{lem}[thm]{Lemma}
\newtheorem{prop}[thm]{Proposition}
\theoremstyle{definition}
\newtheorem{defn}[thm]{Definition}
\theoremstyle{remark}
\newtheorem{rem}[thm]{Remark}
\newtheorem*{ex}{Example}
\numberwithin{equation}{section}
     \title[Controllability of Integrodifferential equations]{Approximate Controllability for Nonautonomous Integrodifferential Equations with State-dependent Delay}

     \author{Mamadou Abdoul DIOP}
     \address{Department of Mathematics, Faculty of Applied Sciences and Technology, Gaston Berger University, Senegal.}
     \email{mamadou-abdoul.diop@ugb.edu.sn}
     \author{Mohammed ELGHANDOURI}
     \address{Cadi Ayyad University, Faculty of Sciences Semlalia, Departement of Mathematics, B.P: 2390, Marrakesh, Morocco and EDITE (ED130), IRD, UMMISCO,
     	F-93143, Sorbonne University, Paris,
     	France.}
     \email{medelghandouri@gmail.com}

     \author[K. Ezzinbi]{Khalil EZZINBI$^{*}$}
     \address{Cadi Ayyad University, Faculty of Sciences Semlalia, Departement of Mathematics, B.P: 2390, Marrakesh, Morocco.}
     \email{ezzinbi@uca.ac.ma}

     \date{Month, Day, Year}


\begin{document}
     \begin{abstract}
     We study the existence of mild solutions and the approximate controllability for  nonautonomous integrodifferential equations with state-dependent delay. We assume the approximate controllability of the linear part, then we use the resolvent operators theory to prove the approximate controllability of the nonlinear case. An example of one-dimensional heat equation with memory is given to illustrate our basic results.
     \end{abstract}
     \maketitle
     \section{Introduction} \label{sec 1}
     Integrodifferential equations serve as powerful tools in modeling a wide array of natural occurrences, ranging from electromagnetic waves to mathematical epidemiology, optimal control in economics, and neural networks, among others. In \cite{Volterra}, the author proposed a partial integrodifferential reaction-diffusion equation, expressing the dynamics of certain elastic materials:
     \begin{equation*}
     \dfrac{\partial z(\theta,t)}{\partial t} =\Delta z(\theta,t)+\displaystyle\int_{0}^{t}\phi(t,s)\Delta z(\theta,s)ds+\varphi(\theta,t), \quad (\theta,t)\in \mathbb{R}\times\mathbb{R}^{+},
     \end{equation*}
     where, $\phi$ and $\varphi$ represent suitable functions. Similarly, in \cite{Davis} and \cite{Bloom}, authors investigated the electric displacement field in Maxwell Hopkinson dielectric using this linear partial integrodifferential equation:
     \begin{equation*}
     \dfrac{\partial^{2} z(\theta,t)}{\partial t^{2}}=\dfrac{1}{\eta \epsilon}\Delta z(\theta,t)+\displaystyle\int_{0}^{t}\dfrac{1}{\eta \epsilon}\psi(t-s)\Delta z(\theta,s)ds, \quad (\theta,t)\in\tilde{\Omega}\times[0,T),
     \end{equation*}
     where $\eta,\epsilon\in \mathbb{R}$ and $\psi$ represents a vector of scalar functions. The Rayleigh problem, also known as the Stokes first problem of viscoelasticity, is represented by the following integrodifferential equation:
     \begin{equation*}
     \begin{array}{l}
     \dfrac{\partial z(\theta,t)}{\partial t}=\displaystyle\int_{0}^{t}\Delta z(\theta,\tau)da(\tau)+h(\theta,t), \quad (\theta,t)\in [0,1]\times \mathbb{R}^{+}.
     \end{array}
     \end{equation*}
     where $a:\mathbb{R}^{+}\to \mathbb{R}$ is a function of bounded variation on each compact interval of $\mathbb{R}^{+}$ with $a(0)=0$. This problem is a typical example of one-dimensional problems in viscoelasticity, like torsion of a rod, simple shearing motions and simple tension, see \cite{Pruss}. A simple control system of integrodifferential equation is that of the electrical RLC circuit, see \cite[Eq (2.2)]{Bishop}. In \cite{Ma}, the authors used some delayed integrodifferential equations to study the dynamic of some epidemiological systems, see \cite[page 685, Eqs (11f)-(11g)]{Ma}. Another motivation comes from biological sciences, physics and other domains, such as elasticity, dynamics populations, forecasting human populations, torsion of a wire, radiation transport, Bernoulies problems, oscillating magnetic field, mortality of equipment problems, inverse problems of reaction diffusion equations, among others, as explored in \cite{Laksh,Pruss,rivera,Volterra1}. In general, this class of equations is overly complicated because many methods used to study classical differential equations do not remain valid for studying integrodifferential equations. Despite their complexity, numerous studies have delved into understanding their solutions and behaviors, as referenced in \cite{Balachandran,dieye,Diop 1,Elghandouri2,Ezzinbi,Ezzinbi 1,Ezzinbi 2,Garnier,lu2017null,Pandolfi,Ravichandran}. These equations, though intricate, remain a focal point of research due to their varied applications. The complexity poses challenges, yet numerous studies continue to explore their behaviors, solutions, and practical implications.
     
     State-dependent delay differential equations offer a dynamic framework widely applied in describing various phenomena, as evidenced by research in sources such as \cite{Aiello,Arino,Belair,Hartung,Insperger,Mahaffy}. These equations provide a more nuanced and realistic representation of systems where the delay is contingent upon the system's current state. Despite their applicability, it's notable that in numerous studies, constant time delay equations have been predominantly favored over state-dependent delay equations. This prevalence might stem from the relative simplicity and ease of analysis associated with constant time delay models. However, this inclination towards constant time delay equations restricts the scope and depth of understanding certain natural phenomena, as state-dependent delays may offer a more accurate portrayal of real-world dynamics in various systems. This tendency limits the comprehensive exploration of systems where delays are contingent on the system's state, potentially overlooking crucial aspects of their behavior and leading to a more limited understanding of these complex phenomena.
     
     Controllability is one of the most important fundamental concepts of mathematical control theory, which plays an important role in deterministic and stochastic systems. Its history began with the case of the finite dimension, its extension to the infinite dimension case has known a very important development since the works of Hector Fattorini in 1971, David Russell in 1978 and Jacques-Louis Lions in the late 1970s, see \cite{J. Louis1,J. Louis2}.  In an abstract way, a given control system in a space of states functions $X$ and a space of control functions $U$ by the following differential equation:
     \begin{equation*}
     x'(t)=F(t,x(t),u(t)),\hspace{0.1cm} 0\leq t\leq \tau,\hspace{0.1cm} \text{and}\hspace{0.1cm} x(0)=x_0 ,
     \end{equation*}
     is said to be controllable if it can be brought in a finite time $\tau$ from an arbitrary initial state $x_{0}$ to a prescribed final state $x_{1}$ under the action of a control function $u:[0,\tau]\to U$. If there exists a function $u$ such that $x(\tau,x_0,u)=x_1$ (where $x(\tau,x_0,u)$ is the state value of the system at time $t=\tau$ corresponding to the initial state $x_{0}$ and the control function $u(\cdot)$), we say that the system is exactly controllable on $[0,\tau]$. R. Triggiani \cite{Trigani1,Trigani2} explained that the exact controllability is a stronger concept and rarely satisfied for some hyperbolic systems and therefore it is more appropriate to look for a weaker concept in order to study the controllability of these types of systems, what is called the approximate controllability. The last guarantees that it is possible to control a movement from any point to an arbitrary neighborhood of any other point although the trajectory typically never reaches the specified end point.
     
     Recently, some work have studied the issue of the approximate controllability of dynamics systems under different conditions. For instance, we refer to \cite{chen,Elghandouri,elghandouri2023approximate,Elghandouri 1,35,yong,sakthivel,sakthivel1,wang} and the references therein. In \cite{chen}, the authors studied the existence and the approximate controllability for a class of nonautonomous evolution parabolic equations with nonlocal conditions in Banach spaces by using resolvent operator condition. In \cite{35}, the author used  the theory of linear evolution operators to discuss the approximate controllability of some semilinear nonautonomous evolution systems with state-dependent delay. In \cite{Elghandouri}, the authors used the theory of resolvent operators instead of the theory of semigroups to investigate the existence of mild solutions and the approximate controllability of some nonlinear integrodifferential equations by assuming the approximate controllability of the linear part. Other methods are presented in \cite{Elghandouri 1,yong,sakthivel1,wang} to study the approximate controllability of different types of equations. In this paper, we use resolvent operators theory instead of semigroups theory to study the existence of mild solutions and the approximate controllability of the following nonautonomous integrodifferential equation:
     \begin{equation}
     \left\{\begin{array}{l}
     x^{'}(t)=-A(t)x(t)+\displaystyle\int_{0}^{t}G(t,s)x(s)ds+F(t,x_{\rho(t,x_t)}) +Bu(t), \quad t\in[0,\tau],\\
     x(t)=\phi(t), \quad t\in]-\infty,0],
     \end{array}\right.
     \label{1}
     \end{equation}
     where $\left\{-A(t):\hspace{0.1cm}t\in\mathbb{R}^{+}\right\}$ and $\left\{G(t,s):\hspace{0.1cm} 0\leq s\leq t, \hspace{0.1cm}t\in\mathbb{R}^{+} \right\} $ are families of closed unbounded linear operators on $X$ with a fixed domain $D(A)$, $F:[0,\tau]\times\mathcal{H}_{\alpha}\to X$ is a nonlinear function, $B:U\to X$ is a bounded linear operator, $x_t:]-\infty,0]\to X$, $x_t(\theta)=x(t+\theta)$ belongs to an abstract space $\mathcal{H}_{\alpha}$ that will be specified later and $\rho:[0,\tau]\times\mathcal{H}_{\alpha}\to]-\infty,\tau]$ is a continuous function. Here, $X$ is a separable reflexive Banach space with an uniformly convex dual $X^{*}$, $U$ is a separable Hilbert space, $\mathcal{H}_{\alpha}\subset \mathcal{H}$ where $\mathcal{H}$ is a phase space that will be specified later. The main purpose of this work is to extend the results obtained in \cite{35} for integrodifferential equations with state depend delay in Banach space. We are going to examine this issue by using fractional power operators theory, that is, we are looking to restrict equation \eqref{1} in a Banach space $X_\alpha\subset X$. First, we establish the existence of the mild solution, and then we prove the approximate controllability of our system. Also, an optimal control for equation \eqref{1} will be established by employing some properties of duality mapping and resolvent operator condition.
     
     We stand by $\langle\cdot,\cdot\rangle $ the duality pairing between $X$ and its topological dual $X^{*}$ and by $(\cdot,\cdot)$ the inner product on $U$. The spaces $X$, $X^*$ and $U$ are endowed with the norms $\Vert\cdot\Vert_X$,  $\Vert\cdot\Vert_{X^*}$ and $\Vert\cdot\Vert_U$ respectively. 
     We denote by $\mathcal{PC}([\mu,\eta];X)$ the space of functions $y(\cdot)$ defined from $[\mu,\eta]$ to $X$ such that $y(\cdot)$ is continuous at $t\neq t_k\in[\mu,\eta]$ and left continuous at $t=t_k$ and the right limit $y(t_k^+)$ exists for $k=1,\ldots,m\in \mathbb{N}^{*}$. Space $\mathcal{PC}([\mu,\eta];X)$ is a Banach space endowed with the norm $\Vert \cdot\Vert_{\mathcal{PC}}$ defined by 
     \begin{center}
     	$\Vert y\Vert_{\mathcal{PC}}=\sup\limits_{t\in [\mu,\eta]}\Vert y(t)\Vert_X$ \text{ for } $y\in \mathcal{PC}([\mu,\eta];X)$.
     \end{center} 
     We denote by $\mathcal{L}(X,Y)$ the space of bounded linear operators defined from $X$ to a linear normed space $Y$ endowed with the norm $\Vert\cdot\Vert_{\mathcal{L}(X,Y)}$ defined by 
     \begin{center}
     	$\Vert T\Vert_{\mathcal{L}(X,Y)}=\sup\{\Vert Tz\Vert_{Y}: \hspace{0.1cm } z\in X,\hspace{0.1cm } \Vert z\Vert_{X}=1 \}$ for $T\in\mathcal{L}(X,Y)$.
     \end{center}  
     It is denoted by $\mathcal{L}(X)$ if $Y=X$.

     The paper is organized as follows. In Section \ref{sec 2}, we recall  some useful properties on the theory of resolvent operators, control systems, phase spaces, duality mapping, and the fractional power of $A(t)$. In Section \ref{sec 3}, we study the controllability of the linear part corresponding to equation (\ref{1}). Also, we prove the existence of an optimal control for equation (\ref{1}). In Section \ref{sec 4}, we study the existence of a mild solution of equation (\ref{1}). In Section \ref{sec 5}, assuming that the linear part is approximately controllable, under sufficient conditions, we show the approximate controllability for the whole system. In Section \ref{sec 6}, an example is presented to illustrate our results.
     \section{Preliminaries} \label{sec 2}
     \subsection{Resolvent operators}
     \noindent
     
     \begin{defn} \cite[page: 93]{14}. Let $\{-A(t):\hspace{0.1cm} t\geq 0\}$ be a family of generators of $C_0$-semigroups. $\{-A(t):\hspace{0.1cm}t\geq 0\}$ is called stable if there are  real constants $\tilde{M}\geq 1$ and $\tilde{\beta}\in \mathbb{R}$ such that 
     	\begin{center}
     		$\Vert (-A(s_k)-\lambda I)^{-1}(-A(s_{k-1})-\lambda I)^{-1}\cdots(-A(s_1)-\lambda I)^{-1} \Vert_{\mathcal{L}(X)}\leq \tilde{M}(\lambda-\tilde{\beta})^{-k} $
     	\end{center}
     	for all $\lambda>\tilde{\beta}$ and for every  finite sequence $0\leq s_1\leq s_2\leq \cdots\leq s_k$, $k\in \mathbb{N}^{*}$.
     \end{defn}
     
     Let $\mathcal{F}=BU(\mathbb{R}^{+};X)$  be the space of all bounded uniformly continuous functions defined from $\mathbb{R}^{+}$ to $X$. We denote $Y$ as the Banach space formed by $D(A)$ equipped with the graph norm $\Vert y\Vert_{Y}=\Vert A(0)y\Vert_{X}+\Vert y\Vert_{X}$ for $y\in D(A)$. Let us define the linear operator $\tilde{G}(t)$, $t\geq 0$, mapping from $Y$ to $\mathcal{F}$ as $(\tilde{G}(t)y)(s)=G(t+s,t)y$ for $s\geq0$ and $y\in Y$. Let $D$ be the differentiation operator defined on $\mathcal{F}$ by $Dh=h'$ on a domain $Dom(D)\subset \mathcal{F}$. Then, $D$ serves as the infinitesimal generator of the translation semigroup $(S(t))_{t\geq 0}$ defined on $\mathcal{F}$ by 
     \begin{center}
     	$(S(t)h)(s)=h(t+s)$ for $h\in \mathcal{F}$ and $t,s\geq 0$.
     \end{center}
     
     The following assumptions are needed throughout this work.
     \begin{enumerate}
     	\item[$\bf\huge(R_1)$] $\{-A(t):\hspace{0.1cm}t\geq 0\}$ is a stable family of generators such that $A(t)y$ is strongly continuously differentiable on $\mathbb{R}^{+}$ for each $y\in Y$. In addition $\tilde{G}(t)y$  is strongly continuously differentiable on $\mathbb{R}^{+}$ for each $y\in Y$.
     	\item[$\bf\huge(R_2)$] $\tilde{G}(t)$ is continuous on $\mathbb{R}^{+}$ into $\mathcal{L}(Y,\mathcal{F})$.
     	\item[$\bf\huge(R_3)$] $\tilde{G}(t)Y\subset Dom(D)$ for each $t\geq0$ and $D\tilde{G}(t)$ is continuous on $\mathbb{R}^{+}$ into $\mathcal{L}(Y,\mathcal{F})$.
     \end{enumerate}
     \begin{rem}
     	If $A(t)=A$ and $G(t,s)=G(t-s)$, then $\tilde{G}(t)\in \mathcal{L}(Y,X)$ is constant, which ensures that assumptions $\bf\huge(R_1)$ and $\bf\huge(R_2)$ are satisfied if $-A$ generates a $C_0$-semigroup.
     \end{rem}
     
     The author in \cite{1}, studied the existence of the  resolvent operator of the following nonautonomous equation: 
     \begin{equation}
     \left\{\begin{array}{l}
     x^{'}(t)=-A(t)x(t)+\displaystyle\int_{0}^{t}G(t,s)x(s)ds, \quad t\geq 0\\
     x(0)=x_0\in X.
     \end{array}\right.
     \label{2}
     \end{equation}
     
     In the next, we recall some useful properties on this theory. 
     \begin{defn} \cite[Definition 2.2]{1}. A resolvent operator of equation (\ref{2}) is a bounded linear operator valued function $R(t,s)$ with $0\leq s\leq t$, having the following properties:
     	\begin{enumerate}
     		\item[1)] $R(t,s)$ is strongly continuous in $s$ and $t$, $R(s,s)=I$, for $s\geq 0$ and $\Vert R(t,s)\Vert_{\mathcal{L}(X)}\leq Me^{\beta(t-s)}$, $t\geq s\geq0$, for $M\geq 1$ and $\beta\in \mathbb{R}$.
     		\item[2)] $R(t,s)Y\subset Y$ and $R(t,s)$ is strongly continuous in $s$ and $t$ on $Y$.
     		\item[3)] For each $y\in Y$, $R(t,s)y$ is strongly continuously differentiable in $t$ and $s$. Moreover,
     		\begin{equation*}
     		\left\{\begin{array}{l}
     		\dfrac{\partial R(t,s)y}{\partial t}=-A(t)R(t,s)y+\displaystyle\int_{s}^{t}G(t,r)R(r,s)ydr\\[10pt]
     		\dfrac{\partial R(t,s)y}{\partial s}=R(t,s)A(s)y-\displaystyle\int_{s}^{t}R(t,r)G(r,s)ydr,
     		\end{array} \right.
     		\end{equation*} 
     		with $\dfrac{\partial R(t,s)y}{\partial t}$ and $\dfrac{\partial R(t,s)y}{\partial s}$ are strongly continuous on $0\leq s\leq t$.
     	\end{enumerate}
     \end{defn}
     \begin{lem} \cite[Theorem 2.4]{1}. If $A(t)=A$ and $G(t,s)=G(t-s)$ for $t\geq 0$ and $t\geq s\geq 0$ respectively, then $R(t,s)=R(t-s)$ for each $t\geq s\geq 0$.
     \end{lem}
     
     The following Theorem ensures the existence of the resolvent operator for equation (\ref{2}).
     \begin{thm} \cite[Theorems 2.3 and 3.7]{1}. \itshape Assume that  $\bf\huge(R_1)$-$\bf\huge(R_3)$ hold. Then, equation $(\ref{2})$
     	has a unique resolvent operator.
     	\label{thm 4}
     \end{thm}
     
     In the whole of this work, we assume that $\bf\huge(R_1)$-$\bf\huge(R_3)$ hold. We consider the following evolution equation:
     \begin{equation}
     \left\{\begin{array}{l}
     x^{'}(t)=-A(t)x(t)+\displaystyle\int_{0}^{t}G(t,s)x(s)ds+f(t), \quad t\geq 0\\
     x(0)=x_0\in X,
     \end{array}\right.
     \label{2.2}
     \end{equation}
     where $f\in \mathbb{L}^{1}_{loc}(\mathbb{R}^{+};X)$. 
     \begin{defn} \cite{1} A function $x:[0,+\infty[\to X$ is called a strict solution of equation \eqref{2.2} if $x\in \mathcal{C}^{1}\left([0,+\infty[;X\right)\cap \mathcal{C}\left([[0,+\infty[;D(A)\right)$, and $x$ satisfies equation \eqref{2.2}.
     \end{defn}
     
     The following Theorem guarantees the existence of a strict solution for equation \eqref{2.2}.
     \begin{thm} \cite[Corollary 3.8]{1}. \itshape If $x_0\in D(A)$ and $f\in \mathcal{C}^{1}(\mathbb{R}^{+};X)$, then equation $(\ref{2})$ has a unique strict solution.
     \end{thm}
     \begin{thm}
     	If $x$ is a strict solution of equation \eqref{2.2}, then $x$ is given by the following variation of constants formula
     	\begin{equation}
     	x(t)=R(t,0)x_0+\displaystyle\int_{0}^{t}R(t,s)f(s)ds, \quad t\in\mathbb{R}^{+},
     	\label{4}
     	\end{equation}
     	where $\{R(t,s):\hspace{0.1cm}0\leq s\leq t\}$ is the resolvent operator of equation $\eqref{2}$.
     \end{thm}
     \begin{defn} A function $x:[0,+\infty[\to X$ satisfying $(\ref{4})$  is called a mild solution of equation (\ref{2.2}).
     	\label{defn 1}
     \end{defn}
     
     We assume that.
     \begin{enumerate}
     	\item[$\bf\huge(C_1)$] $-A(t)$ is closed and the domain $D(A(t))=D(A)$ is independent of $t$ and is dense in $X$.
     	\item[$\bf\huge(C_2)$] For each $t\geq 0$, the map $R(\lambda,-A(t))=(\lambda I+ A(t))^{-1}$ exists for each $\lambda\in \mathbb{C}$ with $Re(\lambda)\leq 0$ and there exists $K_1>0$ such that
     	\begin{center}
     		$	\Vert R(\lambda,A(t)) \Vert_{\mathcal{L}(X)}\leq \dfrac{K_1}{\vert \lambda \vert+1} $.
     	\end{center}
     	\item[$\bf\huge(C_3)$]  There exist constants $K_2>0$ and $0<\gamma\leq 1$ such that 
     	\begin{center}
     		$\Vert (A(s_1)-A(s_2))A(s_3)^{-1}\Vert_{\mathcal{ L}(X)}\leq K_2\vert s_1-s_2\vert^{\gamma}$ \text{ for } $s_1\geq 0$, $s_2\geq 0$ and $s_3\geq 0$.
     	\end{center}
     \end{enumerate}
     \begin{lem} \cite[Chapter 5, Theorem 6.1]{28}. \itshape Assume that $\bf\huge(C_1)$-$\bf\huge(C_3)$ hold. Then, there exists a unique  evolution system $\{U(t,s):\hspace{0.1cm}0\leq s\leq t\}$ generated by the family $\{-A(t):\hspace{0.1cm}t\geq0\}$.
     \end{lem}
     \begin{lem}\cite[Lemma 2.10]{diop}. \itshape Assume that $\bf\huge(C_1)$-$\bf\huge(C_3)$ are satisfied. Let $\{U(t,s):\hspace{0.1cm}0\leq s\leq t\}$ be the unique evolution system generated by  $\{-A(t):\hspace{0.1cm}t\geq0\}$ and $\{R(t,s):\hspace{0.1cm}0\leq s\leq t\}$ be the unique resolvent operator of equation $(\ref{2})$. Then,
     	\begin{equation}
     	R(t,s)x=U(t,s)x+\displaystyle\int_{s}^{t}U(t,r)Q(r,s)x dr, \quad 0\leq s\leq t, \quad x\in X,
     	\label{equ 2.3}
     	\end{equation}
     	where
     	\begin{center}
     		$Q(r,s)x=G(r,r)\displaystyle\int_{s}^{r}R(v,s)xdv-\displaystyle\int_{s}^{r}\dfrac{\partial G(r,v)}{\partial v}\displaystyle\int_{s}^{v}R(w,s)xdwdv$,
     	\end{center}
     	for $r\geq s\geq 0$, and $x\in X$. Moreover, $\{Q(r,s)x:\hspace{0.1cm}r\geq s\geq 0\}$ is uniformly bounded on bounded intervals and for each $x\in X$, $Q(\cdot,\cdot)x\in\mathcal{C}(\mathbb{R}^{+}\times\mathbb{R}^{+};X)$.
     	\label{lemma 3}
     \end{lem}
     \begin{lem}
     	Assume that $\bf\huge(C_1)$-$\bf\huge(C_3)$ are satisfied. Let $\{U(t,s):\hspace{0.1cm}0\leq s\leq t\}$ be the unique evolution system generated by  $\{-A(t):\hspace{0.1cm}t\geq0\}$ and $\{R(t,s):\hspace{0.1cm}0\leq s\leq t\}$ be the unique resolvent operator of equation $(\ref{2})$. Then,
     	\begin{equation*}
     	U(t,s)x=R(t,s)x+\displaystyle\int_{s}^{t}R(t,r)P(r,s)x dr, \quad 0\leq s\leq t, \quad x\in X,
     	%\label{equ 2.5}
     	\end{equation*}
     	where
     	\begin{center}
     		$P(r,s)x=-G(r,r)\displaystyle\int_{s}^{r}U(v,s)xdv+\displaystyle\int_{s}^{r}\dfrac{\partial G(r,v)}{\partial v}\displaystyle\int_{s}^{v}U(w,s)xdwdv$,
     	\end{center}
     	for $r\geq s\geq 0$, and $x\in X$. Moreover, $\{P(r,s)x:\hspace{0.1cm}r\geq s\geq 0\}$ is uniformly bounded on bounded intervals and for each $x\in X$, $P(\cdot,\cdot)x\in\mathcal{C}(\mathbb{R}^{+}\times\mathbb{R}^{+};X)$.
     	\label{lemma 4}
     \end{lem}
     \begin{proof} Let $y\in X$ be fixed, and $0\leq  s\leq  t$. Consider the following evolution equation:
     	\begin{equation}
     	\left\{\begin{array}{l}
     	x'(t)=-A(t)x(t)\\
     	\hspace{0.7cm}	= -A(t)x(t)+\displaystyle\int_{s}^{t}G(t,r)x(r)dr-\displaystyle\int_{s}^{t}G(t,r)x(r)dr\\
     	\hspace{0.7cm}	= -A(t)x(t)+\displaystyle\int_{s}^{t}G(t,r)x(r)dr+L(t)\\
     	x(s)\hspace{0.1cm}=y,
     	\end{array}\right.
     	\label{equation 2.6}
     	\end{equation}
     	where $L(t)=-\displaystyle\int_{s}^{t}G(t,r)x(r)dr$. If $x(\cdot)$ is a strict solution of equation \eqref{equation 2.6}, then
     	\begin{center}
     		$x(t)=R(t,s)y+\displaystyle\int_{s}^{t}R(t,r)L(r)dr$ \text{ for } $t\geq s\geq 0$.
     	\end{center}
     	Moreover,
     	\begin{eqnarray*}
     		x'(t)&=& \dfrac{\partial R(t,s)y}{\partial t}+ \dfrac{\partial}{\partial t}\left(\displaystyle\int_{s}^{t}R(t,r)L(r)dr\right)\\
     		&=& -A(t)R(t,s)y+\!\!\displaystyle\int_{s}^{t}G(t,r)R(r,s)ydr\\
     		& & +\displaystyle\int_{s}^{t}\left[-A(t)R(t,r)L(r)+\displaystyle\int_{r}^{t}G(t,u)R(u,r)L(r)du\right]dr+L(t)
     	\end{eqnarray*}
     \begin{eqnarray*}
     		&=& -A(t)\left[R(t,s)y+\displaystyle\int_{s}^{t}R(t,r)L(r)dr \right]\\
     		& & +\displaystyle\int_{s}^{t}\left[\displaystyle\int_{r}^{t}G(t,u)R(u,r)L(r)du+G(t,r)R(r,s)y\right]dr +L(t).
     	\end{eqnarray*}
     	By Fubini Theorem, we obtain
     	\begin{eqnarray*} 
     		x'(t)&=& -A(t)x(t)+\displaystyle\int_{s}^{t}G(t,v)\left[R(v,s)y+\displaystyle\int_{s}^{v}R(v,r)L(r)dr \right]dv+L(t)\\
     		&=& -A(t)x(t)+\displaystyle\int_{s}^{t}G(t,v)x(v)dv+L(t).
     	\end{eqnarray*}
     	Note that $x(t)=U(t,s)y$. Let $H(t)y=\displaystyle\int_{s}^{t}U(v,s)y dv$, then
     	\begin{eqnarray*}
     		U(t,s)y&=& R(t,s)y+\displaystyle\int_{s}^{t}R(t,r)L(r)dr\\
     		&=& R(t,s)y+\displaystyle\int_{s}^{t}R(t,r)\left[-\displaystyle\int_{s}^{r}G(r,v)x(v)dv \right]dr\\
     		&=& R(t,s)y+\displaystyle\int_{s}^{t}R(t,r)\left[-\displaystyle\int_{s}^{r}G(r,v)U(v,s)ydv \right]dr\\
     		&=& R(t,s)y+\displaystyle\int_{s}^{t}R(t,r)\left[-\displaystyle\int_{s}^{r}G(r,v)H'(v)ydv \right]dr\\
     		&=& R(t,s)y+\displaystyle\int_{s}^{t}R(t,r)P(r,s)ydr,
     	\end{eqnarray*}
     	where
     	\begin{eqnarray*}
     		P(r,s)y&=& -\displaystyle\int_{s}^{r}G(r,v)H'(v)ydv\\
     		&=& -\left[ G(r,r)H(r)y-G(r,s)H(s)y-\displaystyle\int_{s}^{r}\dfrac{\partial G(r,v)}{\partial v}H(v)ydv\right]\\
     		&=& -G(r,r)H(r)y+\displaystyle\int_{s}^{r}\dfrac{\partial G(r,v)}{\partial v}H(v)ydv\\
     		&=& -G(r,r)\displaystyle\int_{s}^{r}U(v,s)ydv+\displaystyle\int_{s}^{r}\dfrac{\partial G(r,v)}{\partial v}\displaystyle\int_{s}^{v}U(w,s)ydwdv.
     	\end{eqnarray*}
     	By making minor changes and employing similar reasoning as in \cite[Lemma 2.8]{diop0}, we get that $\{P(r,s)x:\hspace{0.1cm}r\geq s \geq 0\}$ is uniformly bounded on bounded intervals and for each $x\in X$, $P(\cdot,\cdot)x\in\mathcal{C}(\mathbb{R}^{+}\times\mathbb{R}^{+};X)$.
     \end{proof}
     \begin{thm} \label{thm 3} \itshape Assume that $\bf\huge(C_1)$-$\bf\huge(C_3)$ are satisfied. Let $\{U(t,s):\hspace{0.1cm}0\leq s\leq t\}$ be the unique evolution system generated by  $\{-A(t):\hspace{0.1cm}t\geq0\}$ and $\{R(t,s):\hspace{0.1cm}0\leq s\leq t\}$ be the unique resolvent operator of equation $(\ref{2})$. Then, $U(t,s)$ is compact for $t-s>0$ if and only if $R(t,s)$ is compact for $t-s>0$.
     \end{thm}
     \begin{proof}  Assume that $U(t,s)$ is compact for $t-s>0$. Let $\Omega$ be a bounded set of $X$. We show that $R(t,s)\Omega$ is relatively compact whenever $t-s>0$. According to (\ref{equ 2.3}), it is sufficient to show   that 
     	\begin{center}
     		$	\left\{ \displaystyle\int_{s}^{t} U(t,r)Q(r,s)x dr:\hspace{0.1cm} x\in \Omega \right\}$,
     	\end{center}
     	is relatively compact in $X$. Let $\varepsilon>0$ be small enough. Note that
     	\begin{eqnarray*}
     		\displaystyle\int_{s}^{t} U(t,r)Q(r,s)x dr &=& U(t,t-\varepsilon) \displaystyle\int_{s}^{t-\varepsilon} U(t-\varepsilon,r)Q(r,s)x dr  +\displaystyle\int_{t-\varepsilon}^{t} U(t,r)Q(r,s)x dr.
     	\end{eqnarray*}
     	Since, $x\to  \displaystyle\int_{s}^{t-\varepsilon} U(t-\varepsilon,r)Q(r,s)x dr=R(t-\varepsilon,s)x-U(t-\varepsilon,s)x$ is  continuous. Then, 
     	\begin{center}
     		$\left\{ U(t,t-\varepsilon) \displaystyle\int_{s}^{t-\varepsilon} U(t-\varepsilon,r)Q(r,s)x dr:\hspace{0.1cm} x\in \Omega \right\}$,
     	\end{center}
     	is relatively compact in $X$. Since,
     	\begin{center}
     		$\Vert \displaystyle\int_{t-\varepsilon}^{t} U(t,r)Q(r,s)x dr \Vert \leq \bar{b} \varepsilon$,\quad $x\in \Omega$,
     	\end{center}
     	for some constant $\bar{b}>0$. Consequently, the set 
     	\begin{center}
     		$\left\{\displaystyle\int_{t-\varepsilon}^{t} U(t,r)Q(r,s)x dr :\hspace{0.1cm} x\in \Omega \right\}$,
     	\end{center}
     	is totally bounded in $X$, and we deduce that
     	\begin{center}
     		$	\left\{ \displaystyle\int_{s}^{t} U(t,r)Q(r,s)x dr:\hspace{0.1cm} x\in \Omega \right\}$,
     	\end{center}
     	is relatively compact in $X$. As a consequence, $R(t,s)$ is compact for $t-s>0$. Conversely, we use Lemma \ref{lemma 4}.
     \end{proof}
     \subsection{Fractional powers of $A(t)$}
     \noindent
     
     If $\bf\huge(C_1)$-$\bf\huge(C_3)$ hold, then the following integral
     \begin{equation}
     A^{-\alpha}(t)=\dfrac{1}{\Gamma(\alpha)}\displaystyle\int_{0}^{+\infty}s^{\alpha-1}T_t(s)ds,
     \label{15}
     \end{equation}
     exists for each $0<\alpha< 1$ and $t\geq 0$, where $(T_t(s))_{s\geq 0}$ is the analytic semigroup generated by $-A(t)$ and $\Gamma(\cdot)$ is the Gamma function. Note that conditions $\bf\huge(C_1)$ and $\bf\huge(C_2)$ are sufficient to get that $-A(t)$ generates an analytic semigroup on $X$ \cite{28}.
     Furthermore, the linear map given by (\ref{15}) is bounded and satisfying 
     \begin{center}
     	$A^{-\alpha}(t)A^{-\beta}(t)=A^{-(\alpha+\beta)}(t)$ \text{ for }  $t\geq0$  \text{ and } $0<\alpha\leq \beta <1$ \text{ in which } $\alpha+\beta<1$.
     \end{center}
     Thus, we can define $A^{\alpha}(t)$ by $A^{\alpha}(t)=(A^{-\alpha}(t))^{-1}$ which is a closed operator with a dense domain $D(A^{\alpha}(t))$ in $X$. For more details, see \cite[page: 69]{28}. 
     
     In the next, we denote by $X_{\alpha}(t)$ the Banach space formed by $D(A^{\alpha}(t))$ equipped with the norm $\Vert \cdot\Vert_{\alpha, t}$ defined by $\Vert x\Vert_{\alpha, t}=\Vert A^{\alpha}(t)x\Vert_X$ for each $x\in D(A^{\alpha}(t))$, and $\mathcal{  C}_{\alpha}:=\mathcal{  C}([0,\tau];X_{\alpha}(t_0))$ the space of continuous function from $[0,\tau]$ to $X_{\alpha}(t_0)$
     for a fixed $t_0\in[0,\tau]$ endowed with the norm $\Vert \cdot\Vert_{\mathcal{C}_\alpha}$ defined by:
     \begin{center}
     	$\Vert x\Vert_{\mathcal{C}_\alpha}=\sup\limits_{s\in[0,\tau]}\Vert x(s)\Vert_{\alpha,t_0}$, \text{ for } $x\in \mathcal{C}_\alpha $.
     \end{center}
     
     In the rest, we assume that $\bf\huge(C_1)$-$\bf\huge(C_3)$ are true.
     \begin{lem} \label{lemma 10} There is $N_\beta=N_\beta(\tau)>0$ such that 
     	$\Vert A^{\beta}(t)R(t,s)\Vert_{\mathcal{ L}(X)}\leq \frac{N_\beta}{( t-s)^{\beta}}$
     	for  $0\leq s<t\leq\tau$, and $0<\beta< 1$.
     \end{lem}
     \begin{proof}  Let  $N_\beta^{'}=N_\beta^{'}(\tau)>0$ be such that $\Vert A^{\beta}(t)U(t,s)\Vert_{\mathcal{ L}(X)}\leq \dfrac{N_\beta^{'}}{( t-s)^{\beta}}$, for $(s<t)$, $t,s\in[0,\tau]$ (see \cite[Theorm 5.2.22]{Ahmed N. U.}). According to Lemma \ref{lemma 3}, we can affirm that $A^{\beta}(t)R(t,s)\in \mathcal{L}(X)$, moreover for each $x\in D(A^{\beta}(t))$, we have
     	
     	\begin{eqnarray*}
     		\Vert A^{\beta}(t)R(t,s)x\Vert_X
     		&\leq & \Vert A^{\beta}(t)U(t,s)x\Vert_X + \Vert \displaystyle\int_{s}^{t} A^{\beta}(t)U(t,r)Q(r,s)x dr\Vert_X\\
     		&\leq & \dfrac{N^{'}_\beta}{(t-s)^{\beta}}\Vert x\Vert_X + \displaystyle\int_{s}^{t} \dfrac{N^{'}_\beta}{(t-r)^{\beta}} \Vert Q(r,s) x\Vert_X dr\\
     		&\leq & \dfrac{N^{'}_\beta}{(t-s)^{\beta}}\Vert x\Vert_X + N^{'}_\beta\sup\limits_{0\leq s\leq r\leq \tau} \Vert Q(r,s) \Vert_{\mathcal{  L}(X)}\displaystyle\int_{s}^{t} \dfrac{1}{(t-r)^{\beta}}dr\Vert x\Vert_X\\
     		&=& \dfrac{N^{'}_\beta}{(t-s)^{\beta}}\Vert x\Vert_X + N^{'}_\beta\sup\limits_{0\leq s\leq r\leq \tau} \Vert Q(r,s) \Vert_{\mathcal{  L}(X)}\displaystyle\int_{0}^{t-s} \dfrac{1}{r^{\beta}}dr\Vert x\Vert_X\\
     		&\leq & \dfrac{N^{'}_\beta}{(t-s)^{\beta}}\Vert x\Vert_X + N^{'}_\beta\sup\limits_{0\leq s\leq r\leq \tau} \Vert Q(r,s) \Vert_{\mathcal{  L}(X)}\dfrac{(t-s)^{1-\beta}}{1-\beta}\Vert x\Vert_X\\
     		&\leq & \left(1 +\sup\limits_{0\leq s\leq r\leq \tau} \Vert Q(r,s) \Vert_{\mathcal{  L}(X)}\tau \right)\dfrac{N^{'}_\beta}{(1-\beta)(t-s)^{\beta}}\Vert x\Vert_X.
     	\end{eqnarray*}
     	Since $D(A^{\beta}(t))$ is dense in $X$, it follows that
     	\begin{center}
     		$\Vert A^{\beta}(t)R(t,s)\Vert_{\mathcal{ L}(X)}\leq \dfrac{N_\beta}{( t-s)^{\beta}}$, 
     	\end{center}
     	where $$N_\beta=\left(1 +\sup\limits_{0\leq s\leq r\leq \tau} \Vert Q(r,s) \Vert_{\mathcal{  L}(X)}\tau \right)\dfrac{N^{'}_\beta}{1-\beta}.$$
     \end{proof}
     
     In the next, we assume the following assumption.
     \begin{enumerate}
     	\item[$\bf\huge{(H_0})$] The operators $A^{\alpha}(v)$ and $R(t,s)$ commute for every $0<\alpha\leq 1$, that is 
     	\begin{center}
     		$A^{\alpha}(v)R(t,s)y=R(t,s)A^{\alpha}(v)y$ \text{ for } $v,t,s\in[0,\tau]$ \text{ and } $y\in D(A^{\alpha}(v))$.
     	\end{center}
     	Moreover, there exists $C_{\alpha,\beta}\equiv C_{\alpha,\beta}(\tau)>0$ such that 
     	\begin{center}
     		$\Vert A^{\alpha}(t) A^{-\beta}(s)\Vert_{\mathcal{ L}(X)}\leq C_{\alpha,\beta}$ \text{ for } $t,s\in[0,\tau]$ \text{ and } $0<\alpha<\beta<1$.
     	\end{center}
     \end{enumerate}
     
     The following Theorem will play a crucial role in the rest of this work.
     \begin{thm}\itshape Assume that $\bf\huge(H_0)$ hold. Then, there is $\bar{C}_{\alpha,\beta}(\varepsilon)\equiv\bar{C}_{\alpha,\beta}(\tau,\varepsilon)>0$ such that 
     	\begin{center}
     		$\Vert A^{\alpha}(t_0)[R(t+\varepsilon,s)-R(t+\varepsilon,t)R(t,s)]\Vert_{\mathcal{  L}(X)}\leq \bar{C}_{\alpha,\beta}(\varepsilon)\varepsilon^{1-\beta}$
     	\end{center}
     	for every $0<\varepsilon<t$ and $0\leq s\leq  t\leq \tau$. Moreover, $\bar{C}_{\alpha,\beta}(\varepsilon)\to \bar{C}_{\alpha,\beta}(0)\geq 0$ as $\varepsilon\to 0^{+}$.
     	\label{thm 5}
     \end{thm}
     \begin{proof} Let $x\in X$, $0\leq s\leq  t\leq \tau$ and $0<\varepsilon<t$. By Lemma \ref{lemma 3}, we have
     	\begin{eqnarray*}
     			R(t+\varepsilon,t)R(t,s)x
     		&=& R(t+\varepsilon,t)U(t,s)x+R(t+\varepsilon,t)\displaystyle\int_{s}^{t}U(t,r)Q(r,s)xdr\\
     		&=& U(t+\varepsilon,t)U(t,s)x+U(t+\varepsilon,t)\displaystyle\int_{t}^{t+\varepsilon}U(t,r)Q(r,t)U(t,s)x dr\\
     		&  & + \hspace{0.1cm} U(t+\varepsilon,t)\displaystyle\int_{s}^{t}U(t,r)Q(r,s)xdr\\
     		& & +\displaystyle\int_{t}^{t+\varepsilon}U(t+\varepsilon,r)Q(r,t)\displaystyle\int_{s}^{t}U(t,v)Q(v,s)x dv dr\\
     		&=& U(t+\varepsilon,s)x+\displaystyle\int_{t}^{t+\varepsilon}U(t+\varepsilon,r)Q(r,t)U(t,s)x dr\\
     		&  & + \hspace{0.1cm} \displaystyle\int_{s}^{t}U(t+\varepsilon,r)Q(r,s)xdr\\
     		& & +\displaystyle\int_{t}^{t+\varepsilon}U(t+\varepsilon,r)Q(r,t)\displaystyle\int_{s}^{t}U(t,v)Q(v,s)x dv dr\\
     		&=& R(t+\varepsilon,s)x-\displaystyle\int_{t}^{t+\varepsilon}U(t+\varepsilon,r)Q(r,t)xdr\\
     		& & +\displaystyle\int_{t}^{t+\varepsilon}U(t+\varepsilon,r)Q(r,t)U(t,s)x dr\\
     		&  & + \hspace{0.1cm} \displaystyle\int_{t}^{t+\varepsilon}U(t+\varepsilon,r)Q(r,t)\displaystyle\int_{s}^{t}U(t,v)Q(v,s)x dv dr.
     	\end{eqnarray*}
     	Hence,
     	\begin{eqnarray*}
     		 R(t+\varepsilon,t)R(t,s)x-R(t+\varepsilon,s)x
     		&=& \displaystyle\int_{t}^{t+\varepsilon}U(t+\varepsilon,r)Q(r,t)[U(t,s)x-x] dr\\ &  & + \hspace{0.1cm}\displaystyle\int_{t}^{t+\varepsilon}U(t+\varepsilon,r)Q(r,t)\displaystyle\int_{s}^{t}U(t,v)Q(v,s)x dv dr.
     	\end{eqnarray*}
     	Let  $N_\beta^{'}=N_\beta^{'}(\tau)>0$ be such that $\Vert A^{\beta}(t)U(t,s)\Vert_{\mathcal{ L}(X)}\leq \frac{N_\beta^{'}}{( t-s)^{\beta}}$, for $(s<t)$, $t,s\in[0,\tau]$. Then,
     	\begin{eqnarray*}
     		& & \Vert A^{\alpha}(t_0)[R(t+\varepsilon,t)R(t,s)-R(t+\varepsilon,s)]x\Vert_X \\
     		&\leq & \Vert   A^{\alpha}(t_0) \displaystyle\int_{t}^{t+\varepsilon}U(t+\varepsilon,r)Q(r,t)[U(t,s)x-x] dr\Vert_X\\
     		&  & + \hspace{0.1cm} \Vert A^{\alpha}(t_0) \displaystyle\int_{t}^{t+\varepsilon}U(t+\varepsilon,r)Q(r,t)\displaystyle\int_{s}^{t}U(t,v)Q(v,s)x dv dr\Vert_X
     	\end{eqnarray*}
     	\begin{eqnarray*}
     		&\leq & \displaystyle\int_{t}^{t+\varepsilon} \Vert   A^{\alpha}(t_0)U(t+\varepsilon,r)Q(r,t)[U(t,s)x-x]\Vert_X dr\\
     		&  & + \hspace{0.1cm}  \displaystyle\int_{t}^{t+\varepsilon}\Vert A^{\alpha}(t_0)U(t+\varepsilon,r)Q(r,t)\displaystyle\int_{s}^{t}U(t,v)Q(v,s)x dv\Vert_X dr\\
     		&\leq & C_{\alpha,\beta}N^{'}_\beta\displaystyle\int_{t}^{t+\varepsilon}\dfrac{\Vert Q(r,t)[U(t,s)-I]x\Vert_X}{(t+\varepsilon-r)^{\beta}}dr\\
     		&  & + \hspace{0.1cm} \hspace{0.1cm} C_{\alpha,\beta}N^{'}_\beta\displaystyle\int_{t}^{t+\varepsilon}\dfrac{\Vert Q(r,t)\displaystyle\int_{s}^{t}U(t,v)Q(v,s)x dv\Vert_X}{(t+\varepsilon-r)^{\beta}}dr\\
     		&\leq& C_{\alpha,\beta}N^{'}_\beta (C_1(\varepsilon,\tau)+C_2(\varepsilon,\tau))\left(\displaystyle\int_{t}^{t+\varepsilon}\dfrac{dr}{(t+\varepsilon-r)^{\beta}}\right)\Vert x\Vert_X\\
     		&=& \dfrac{C_{\alpha,\beta}N^{'}_\beta (C_1(\varepsilon,\tau)+C_2(\varepsilon,\tau))}{(1-\beta)}\varepsilon^{1-\beta}.
     	\end{eqnarray*}
     	where
     	\begin{center}
     		$C_1(\varepsilon,\tau)=\sup\left\{\left(\Vert Q(r,t)\Vert_{\mathcal{L}(X)} \Vert[U(t,s)-I]\Vert_{\mathcal{L}(X)}\right)\mid 0\leq s \leq t\leq r \leq\tau+\varepsilon\right\}$
     	\end{center}
     	and 
     	\begin{center}
     		$C_2(\varepsilon,\tau)=\sup\left\{\!\!\left(\Vert Q(r,t)\Vert_{\mathcal{L}(X)}\displaystyle\int_{s}^{t}\!\!\Vert U(t,v)\Vert_{\mathcal{L}(X)} \Vert Q(v,s)\Vert_{\mathcal{L}(X)}dv\right)\mid 0\leq s \leq t\leq r \leq\tau+\varepsilon\right\}$
     	\end{center}
     	Consequently,
     	\begin{center}
     		$\Vert A^{\alpha}(t_0)[R(t+\varepsilon,t)R(t,s)-R(t+\varepsilon,s)]\Vert_{\mathcal{  L}(X)}\leq \bar{C}_{\alpha,\beta}(\varepsilon)\varepsilon^{1-\beta}$,
     	\end{center}
     	where 
     	\begin{center}
     		$\bar{C}_{\alpha,\beta}(\varepsilon)=\dfrac{C_{\alpha,\beta}N^{'}_\beta (C_1(\varepsilon,\tau)+C_2(\varepsilon,\tau))}{(1-\beta)}$.
     	\end{center}
     	Note that
     	\begin{center}
     		$\bar{C}_{\alpha,\beta}(\varepsilon)\to \bar{C}_{\alpha,\beta}(0)=\dfrac{C_{\alpha,\beta}N^{'}_\beta (C_1(0,\tau)+C_2(0,\tau))}{(1-\beta)}\geq 0$ as $\varepsilon\to 0^+$.
     	\end{center}
     \end{proof}
     \subsection{Phase space}
     \noindent
     
     We recall some axioms for the phase space $\mathcal{H}$ introduced in \cite{15}. In details, $\mathcal{H}$ is a linear space of all functions defined from $]-\infty,0]$ into  $X$ equiped with the norm $\Vert \cdot\Vert_{\mathcal{H}}$ satisfying the following axioms:
     \begin{enumerate}
     	\item[$\bf\huge{(A_1)}$] If $x:]-\infty,\mu+\sigma]\to X$, $\sigma>0$, such that $x_\mu\in\mathcal{H}$ and $x_{\mid_{[\mu,\mu+\sigma]}}\in \mathcal{PC}([\mu,\mu+\sigma];X)$. Then, for each $t\in[\mu,\mu+\sigma]$, the following conditions are verified:
     	\begin{enumerate}
     		\item[] (i) $x_t\in \mathcal{H}$.
     		%\item[(ii)] $\Vert x(t)\Vert_X\leq \tilde{H}\Vert x_t\Vert_{\mathcal{H}}$ for some constant $\tilde{H}>0$.
     		\item[] (ii) $\Vert x_t\Vert_{\mathcal{H}}\leq\tilde{K}(t-\mu) \sup\{\Vert x(s)\Vert_{X}:\hspace{0.1cm} \mu\leq s\leq t \}+\tilde{M}(t-\mu)\Vert x_{\mu}\Vert_{\mathcal{H}}$, where $\tilde{K}:\mathbb{R}^{+}\to\mathbb{R}^{+} $ is a continous function,  $\tilde{M}:\mathbb{R}^{+}\to\mathbb{R}^{+} $ is locally bounded and both $\tilde{K}$ and $\tilde{M}$ are independent of $x(\cdot)$.
     	\end{enumerate}
     	\item[$\bf\huge{(A_2)}$] For a function $x(\cdot)$ in $\bf\huge{(A_1)}$, the function $t\to x_t\in \mathcal{H}$ is continuous on $[\mu,\mu+\sigma]$.
     	\item[$\bf\huge{(A_3)}$] The phase space $\mathcal{H}$ is complete.
     \end{enumerate}
     
     We consider the following assumption.
     \begin{enumerate}
     	\item[$\bf\huge(H_1)$] $A^{-\alpha}(t_0)\psi\in \mathcal{  H}$ for $\psi\in \mathcal{  H}$, where function $A^{-\alpha}(t_0)\psi$ is defined by 
     	\begin{center}
     		$(A^{-\alpha}(t_0)\psi)(\theta)=A^{-\alpha}(t_0)\psi(\theta)$, \text{ for } $\psi\in \mathcal{  H}$, \text{ and } $\theta\leq 0.$
     	\end{center}
     \end{enumerate}
     
     We define the set $\mathcal{H}_{\alpha}$ as follows
     \begin{center}
     	$\mathcal{H}_{\alpha}=\{\psi\in\mathcal{H}:\hspace{0.1cm}\psi(\theta)\in X_{\alpha}(t_0)$ for $\theta\leq 0 \text{ and } A^{\alpha}(t_0)\psi\in \mathcal{H} \}$,
     \end{center}
     where the function $A^{\alpha}(t_0)\psi$ is defined by $(A^{\alpha}(t_0)\psi)(\theta)=A^{\alpha}(t_0)\psi(\theta)$ for $\theta\leq 0$.
     For any $\psi\in \mathcal{H}_{\alpha}$, the norm  $\Vert \cdot\Vert_{\mathcal{H}_{\alpha}}$ is defined by $\Vert A^{\alpha}(t_0)\psi(\theta)\Vert_{X}$ instead of $\Vert \psi(\theta)\Vert_{X}$. 
     
     \begin{lem} \cite{Benkhalti} 
     	Assume that $\bf\huge(H_1)$ hold. If $\mathcal{  H}$ satisfies axioms $\bf\huge(A_1)$-$\bf\huge(A_3)$, then the same follows for $\mathcal{H}_{\alpha}$, (i.e), 
     	\begin{enumerate}
     		\item[$\bf\huge(A'_1)$] If $x:]-\infty,\mu+\sigma]\to X_\alpha(t_0)$, $\sigma>0$, such that $x_\mu\in\mathcal{H}_{\alpha}$ and $x_{\mid_{[\mu,\mu+\sigma]}}\in \mathcal{PC}([\mu,\mu+\sigma];X_\alpha(t_0))$. Then, for each  $t\in[\mu,\mu+\sigma]$, the following conditions are satisfied:
     		\begin{enumerate}
     			\item[] (i) $x_t\in \mathcal{H}_{\alpha}$.
     			\item[] (ii) $\Vert x_t\Vert_{\mathcal{H}_{\alpha}}\leq\tilde{K}(t-\mu) \sup\{\Vert x(s)\Vert_{\alpha,t_0}:\hspace{0.1cm} \mu\leq s\leq t \}+\tilde{M}(t-\mu)\Vert x_{\mu}\Vert_{\mathcal{H}_\alpha}$ 
     		\end{enumerate}
     		where $\tilde{K} \text{ and } \tilde{M}$ are the same as defined earlier in $\bf\huge(A_1)$.
     		\item[$\bf\huge(A'_2)$]  For a function $x(\cdot)$ in $\bf\huge{(A'_1)}$, the function $t\to x_t\in \mathcal{H}_{\alpha}$ is continuous on $[\mu,\mu+\sigma]$.
     		\item[$\bf\huge(A'_3)$] The phase space $\mathcal{H}_{\alpha}$ is complete.
     	\end{enumerate}
     	\label{Lemma 6}
     \end{lem}
     
     For any $\psi\in\mathcal{H}_{\alpha}$, the function $\psi_t$, $t\leq 0$ defined as $\psi_t(\theta)=\psi(t+\theta)$, $\theta\in]-\infty,0]$. 
     For the function $\rho:[0,\tau]\times \mathcal{H}_{\alpha}\to ]-\infty,\tau]$, we define
     \begin{center}
     	$\varLambda(\rho^-)=\{\rho(s,\psi):\hspace{0.1cm}\rho(s,\psi)\leq 0, \hspace{0.1cm} (s,\psi)\in[0,\tau]\times \mathcal{H}_{\alpha}\}$.
     \end{center}

     In the sequel, we need the following Lemma.
     \begin{lem} \itshape Let $x:]-\infty,\tau]\to X_\alpha(t_0)$ be a function such that $x_0=\phi$ and $x_{\mid_{[0,\tau]}}\in \mathcal{PC}\left([0,\tau];X_\alpha(t_0)\right)$. Assume that function $t\to \phi_t$ defined from $\varLambda(\rho^{-})$ into $\mathcal{H}_{\alpha}$ is well defined and there exists a continuous bounded function $\varTheta^{\phi}$ defined from $\varLambda(\rho^{-})$ to $\mathbb{R}^{+}$ such that 
     	\begin{center}
     		$\Vert \phi_t\Vert_{\mathcal{H}_{\alpha}}\leq \varTheta^{\phi}(t) \Vert \phi\Vert_{\mathcal{H}_{\alpha}}$.
     	\end{center} Then, 
     	\begin{center}
     		$\Vert x_s\Vert_{\mathcal{H}_{\alpha}}\leq H_2\Vert \phi\Vert_{\mathcal{H}_{\alpha}}+H_3\sup\{\Vert x(\theta)\Vert_{\alpha,t_0}:\hspace{0.1cm} \theta\in[0,\max(0,s)],\hspace{0.1cm}s\in \varLambda(\rho^-)\cup[0,\tau] \}$,
     	\end{center}
     	where
     	\begin{center}
     		$H_2=\sup\limits_{t\in \varLambda(\rho^-)}\varTheta^{\phi}(t)+\sup\limits_{t\in [0,\tau]}\tilde{M}(t)$ \text{ and } $H_3=\sup\limits_{t\in [0,\tau]}\tilde{K}(t)$.
     	\end{center}
     	\label{lem 2.2}
     \end{lem}
     \begin{proof} The result follows from the fact that
     	\begin{center}
     		$\Vert x_s\Vert_{\mathcal{H}_{\alpha}}=\Vert \phi_s\Vert_{\mathcal{H}_{\alpha}}\leq \varTheta^{\phi}(s) \Vert \phi\Vert_{\mathcal{H}_{\alpha}}$ \text{ for } $s\in \varLambda(\rho^{-})$,
     	\end{center}
     	and
     	\begin{center}
     		$\Vert x_s\Vert_{\mathcal{H}_{\alpha}}\leq\tilde{K}(s) \sup\{\Vert x(\theta)\Vert_{\alpha,t_0}:\hspace{0.1cm} 0\leq \theta\leq s \}+\tilde{M}(s)\Vert \phi\Vert_{\mathcal{H}_\alpha}$ \text{ for } $s\in[0,\tau]$.
     	\end{center}
     \end{proof}
     \subsection{Duality mapping}
     \noindent
     
     In the sequel, we recall some useful properties on the duality mapping theory.
     \begin{defn} \cite[Definition 2.1]{17}. A normed linear space $(E,\Vert\cdot\Vert_{E})$ is called smooth if for every $x\in E$ with $\Vert x\Vert_{E}=1$, there exists a unique $x^*\in (E^*,\Vert \cdot\Vert_{E^*})$ such that $\Vert x^*\Vert_{E^*}=1$ and $\langle x,x^*\rangle=\Vert x\Vert_{E}$.
     \end{defn}
     \begin{defn} \cite{18} A normed linear space $(E,\Vert\cdot\Vert_{E})$ is called strictly convex if $\Vert x+y\Vert_{E}=\Vert x\Vert_{E}+\Vert y\Vert_{E}$ implies that $x=c\cdot y$ for some constant $c>0$ whenever $x\neq0$ and $y\neq0$.
     \end{defn}
     \begin{ex} \cite[Part 3, Chap 1]{19}. Let $\Omega\subseteq \mathbb{R}$. The spaces $(\mathbb{L}^{p}(\Omega),\Vert\cdot\Vert_{\mathbb{L}^{p}})$, $1<p<+\infty$ are strictly convex and smooth.
     \end{ex}
     \begin{lem} \cite[Corollary 1.102]{2}. \itshape A reflexive normed space is smooth (strictly convex) if and only if its dual is strictly convex (smooth).
     \end{lem}
     \begin{lem} \cite[Theorem 1.105]{2}. \itshape Let $E$ be a reflexive Banach space. Then, there is an equivalent norm on $E$, such that under this new norm, $E$ and $E^*$ are strictly convex, that is $E$ and $E^*$ are simultaneously smooth and strictly convex.
     	\label{lemma 1}
     \end{lem}
     
     Let $\mathcal{P}(X)$ be the collection of all subsets of $X$.
     \begin{defn} \cite[Definition 1.99]{2}. The map $J:X\to \mathcal{P}(X^*)$ defined by 
     	\begin{center}
     		$J(x)=\{x^*\in X^*:\hspace{0.1cm} \langle x,x^*\rangle=\Vert x\Vert_{X}^{2}=\Vert x^*\Vert_{X^*}^{2}  \}$
     	\end{center}
     	is called the duality mapping of $X$.
     \end{defn}
     
     By Corollary 1.53 from \cite{2}, we can see that $J(x)\neq \emptyset$, for each $ x\in X$, hence $J$ is well defined. Furthermore, $J(x)$ is a convex set of $X^*$ for every $x\in X$. Since $J(x)$ is bounded and weakly closed for every $x\in X$, it follows by Corollary  1.70 from \cite{2}, that $J(x)$ is weakly compact for every $x\in X$.
     \begin{thm} \cite[Remark 1.100]{2}. \itshape
     	The duality mapping $J$ is a single-valued if and only if the normed space $X$ is smooth.
     \end{thm} 
     \begin{rem}
     	If $X $ is a real Hilbert space, then the duality mapping is exactly the canonical isomorphism given by the Riesz Theorem. 
     \end{rem}
     
     \begin{lem} \cite{2} \itshape Assume that $X$ and $X^*$ are simultaneously smooth and strictly convex, then $J$ is bijective, demicontinuous (i-e continuous from $X$ with a strongly topology into $X^*$ with the weak topology) and strictly monotonic. Moreover, $J^{-1}:X^{*}\to X$ is also a duality mapping.
     \end{lem}
     
     Without loss of generality, by Lemma \ref{lemma 1}, we can assume that $X$ and $X^{*}$ are simultaneously smooth and strictly convex with the norms $\Vert\cdot\Vert_{X}$ and $\Vert\cdot\Vert_{X^*}$ respectively.
     
     \begin{defn}
     	A function $x:]-\infty,\tau]\to X_\alpha(t_0)$ is called a mild solution of equation $(\ref{1})$ if
     	\begin{equation*}
     	\left\{\begin{array}{l}
     	x(t)= R(t,0)\phi(0)+\displaystyle\int_{0}^{t}R(t,s)[Bu(s)+F(s,x_{\rho(s,x_s)})]ds, \text{ for } t\in[0,\tau],\\ 
     	x(t)=\phi(t), \text{ for } t\in]-\infty,0].
     	\end{array}\right.
     	\end{equation*}
     \end{defn}
     \begin{defn}
     	Equation $(\ref{1})$ is called approximately controllable on $[0,\tau]$ if for every $\varepsilon>0$ and $d\in X$, there exists $u_\varepsilon(\cdot)\in\mathbb{L}^{2}([0,\tau];U)$ such that 
     	\begin{center}
     		$\Vert x(\tau,u_\varepsilon)-d\Vert<\varepsilon$,
     	\end{center}
     	where $x(\tau,u_\varepsilon)$ is the state value of equation $(\ref{1})$ at time $t=\tau$ corresponding to the control function $u_\varepsilon(\cdot)$.
     \end{defn}
     
     Let 
     \begin{center}
     	$ M_\tau=\sup\{ \Vert R(t,s)\Vert_{\mathcal{L}(X)}:\hspace{0.1cm} 0\leq s\leq t\leq \tau \}$ \text{ and } $M_B=\Vert B\Vert_{\mathcal{  L}(U;X)}$.
     \end{center}
 \section{Controllability  results for the nonautonomous linear part} \label{sec 3}
 \noindent
 
 To prove the approximate controllability of equation (\ref{1}), we establish some controllability results for the following linear equation:
 \begin{equation}
 \left\{\begin{array}{l}
 x^{'}(t)=-A(t)x(t)+\displaystyle\int_{0}^{t}G(t,s)x(s)ds+Bu(t), \quad t\in[0,\tau]\\
 x(0)=x_0\in X.
 \end{array}\right.
 \label{5}
 \end{equation} 
 
 Define the following operators:
 \begin{center}
 	$\begin{array}{ccccc}
 	\mathcal{L}_\tau & : & \mathbb{L}^{2}([0,\tau]; U) & \longrightarrow & X \\
 	& & u & \mapsto &\displaystyle\int_{0}^{\tau}R(\tau,t)Bu(t)dt,
 	\end{array}$
 \end{center}
 \begin{center}
 	$\begin{array}{ccccc}
 	\mathcal{B}_\tau & : & X^* & \longrightarrow & X \\
 	& & x^* & \mapsto &\displaystyle\int_{0}^{\tau}R(\tau,t)BB^*R(\tau,t)^*x^*dt,
 	\end{array}$
 \end{center}
 and 
 \begin{center}
 	$\mathcal{R}(\lambda,\mathcal{B}_\tau )x=(\lambda I+\mathcal{B}_\tau J)^{-1}x$ \text{ for } $\lambda>0$ \text{ for } $x\in X$,
 \end{center}
 provided that this inverse exits.
 
 Our first step is to reformulate an optimal control for the linear equation (\ref{5}) by minimizing the following cost function:
 \begin{equation*}
 \mathcal{G}_\lambda(x,u)=\Vert x(\tau)-d\Vert_{X}^{2}+\lambda\displaystyle\int_0^\tau\Vert u(t)\Vert_U^2dt,
 \end{equation*}
 where $x(\cdot)$ is the mild solution of the linear control system (\ref{5}) corresponding to the control function $u\in\mathbb{L}^{2}([0,\tau];U)$ with $d\in X$ and $\lambda>0$. We recall that the admissible control set is given by $$\mathcal{ U}_{ad}=\mathbb{L}^{2}([0,\tau];U).$$
 Since $B$ is a bounded linear operator, it follows by Definition \ref{defn 1} that the unique mild solution of equation (\ref{5}) corresponding to a control function $u\in\mathcal{  U}_{ad}$ is given by 
 \begin{center}
 	$x(t)=R(t,0)x_0+\displaystyle\int_{0}^{t}R(t,s)Bu(s)ds$ \text{ for } $t\in[0,\tau]$.
 \end{center}
 The following set is called the admissible class of equation (\ref{5}):
 \begin{center}
 	$\mathcal{A}_{ad}=\{(x,u):\hspace{0.1cm} x $ is the unique mild solution of equation (\ref{1}) corresponding to $u\in \mathcal{ U}_{ad} \}$.
 \end{center}
 
 Under the above definitions, the optimal control problem is formulated as follows:
 \begin{equation}
 \min\{\mathcal{G}_\lambda(x,u):\hspace{0.1cm} (x,u)\in \mathcal{A}_{ad}\}.
 \label{7}
 \end{equation}
 \begin{defn} A solution of problem $(\ref{7})$ is called an optimal solution of equation $(\ref{5})$.
 	
 \end{defn}
 \begin{defn} Let $(x^*,u^*)$ be an optimal solution of $(\ref{5})$. The control function $u^*$ is called an optimal control of equation $(\ref{5})$.
 \end{defn}
 
 Let $\Phi:\mathbb{L}^{2}([0,\tau];U)\to \mathcal{  C}([0,\tau];X)$, $g\to \Phi(g)$ be such that
 \begin{center}
 	$\Phi(g)(t)=\displaystyle\int_{0}^{t}R(t,s)Bg(s)ds$ \text{ for } $g\in \mathbb{L}^{2}([0,\tau];U)$ \text{ and } $t\in[0,\tau]$.
 \end{center}
 \begin{lem}
 	Assume that $R(t,s)$ is compact for $t-s>0$, then $\Phi$ is compact.
 	\label{lemma 8}
 \end{lem}
 \begin{proof} Let $\Omega$ be a bounded set in $\mathbb{L}^{2}([0,\tau];U)$. Since $R(t,s)$ is compact for $t-s>0$, then we show that
 	\begin{center}
 		$\Phi(\Omega)(t)=\left\{ \displaystyle\int_{0}^{t}R(t,s)Bg(s)ds; \hspace{0.1cm} g\in\Omega \right\}$,
 	\end{center}
 	is relatively compact for $t\geq 0$. Indeed, for $t=0$, we have $\Phi(\Omega)(0)=\{0\}$ is relatively compact. For $t>0$, we define $\Phi_t:\mathbb{L}^{2}([0,\tau];U)\to X$ by 
 	\begin{center}
 		$\Phi_t(g)=\displaystyle\int_{0}^{t}R(t,s)Bg(s)ds$ \text{ for } $g\in \mathbb{L}^{2}([0,\tau];U)$.
 	\end{center} 
 	To show that $\Phi(\Omega)(t)$ is relatively compact for $t>0$, it is sufficient to show that $\Phi_t$ is compact for all $t>0$. Takes $n\in \mathbb{N}^{*}$ and put $s_k=\frac{kt}{n}$ for $k=0,1,\cdots,n$. Define operators $\Phi_{t,n}\in \mathcal{  L}(\mathbb{L}^{2}([0,\tau];U);X)$ by 
 	\begin{center}
 		$\Phi_{t,n}(g)=\sum\limits_{k=1}^{n}R(t,s_{k-1})B\displaystyle\int_{s_{k-1}}^{s_k}g(s)ds$ \text{ for } $g\in \mathbb{L}^{2}([0,\tau];U)$.
 	\end{center}
 	Since $R(t,s_k)$ is compact, $\Phi_{t,n}$ is compact as well. In the next, we show that $\lim\limits_{n\to +\infty}\Phi_{t,n}=\Phi_t$ with respect to the operator norm. Let $\varepsilon>0$, since $R(t,s)$ is norm-continuous for $t-s>0$, there exists $n_0\in \mathbb{N}^{*}$ such that $\Vert R(t,s)-R(t,s_{k-1})\Vert_{\mathcal{  L}(X)}<\varepsilon$  for all $s\in[s_{k-1},s_k]$, $k=1,2,\cdots,n$ whenever $n\geq n_0$. Hence, for $g\in \mathbb{L}^{2}([0,\tau];U)$, we obtain that   
 	\begin{eqnarray*}
 		\Vert \Phi_{t,n}(g)-\Phi_{t}(g)\Vert_X &=&\Vert\sum\limits_{k=1}^{n} \displaystyle\int_{s_{k-1}}^{s_k}[R(t,s)-R(t,s_{k-1})]Bg(s)ds \Vert_X\\
 		&\leq &  \sum\limits_{k=1}^{n} \displaystyle\int_{s_{k-1}}^{s_k}\Vert[R(t,s)-R(t,s_{k-1})]B\Vert_{\mathcal{  L}(X)}\Vert g(s)\Vert_Xds\\
 		&\leq & \varepsilon \Vert B\Vert_{\mathcal{L}(U,X)} \displaystyle\int_{0}^{t}\Vert g(s)\Vert_Xds\\
 		&< & \varepsilon\Vert B\Vert_{\mathcal{L}(U,X)} \sqrt{t}\Vert g\Vert_{\mathbb{L}^{2}([0,\tau];U)}.
 	\end{eqnarray*}
 	Thus $\Vert \Phi_{t,n}(g)-\Phi_{t}(g)\Vert_{\mathcal{  L}(\mathbb{L}^{2}([0,\tau];U);X)}<\varepsilon \sqrt{t}$. This shows that operator $\Phi_t$ is compact. Consequently,  $\Phi(\Omega)(t)$ is relatively compact for $t>0$. Now, let $0<t'<t^{*}<\tau$ and $g\in \Omega$. Then,
 	\begin{eqnarray*}
 		\Vert \Phi(g)(t^{*})-\Phi(g)(t')\Vert_X & \leq & \displaystyle\int_{0}^{t'}\Vert R(t^{*},s)-R(t',s)\Vert_{\mathcal{L}(X)}\Vert Bg(s)\Vert_X ds\\
 		&  & + \hspace{0.1cm} \displaystyle\int_{t'}^{t^{*}}\Vert R(t^{*},s)Bg(s)\Vert_X ds.
 	\end{eqnarray*}
 	Since $R(t,s)$ is norm continuous for $t-s>0$, then $\lim\limits_{t^{*}\to t'}\Vert \Phi(g)(t^{*})-\Phi(g)(t')\Vert_X=0$ uniformly for $g\in \Omega$. For $t'=0$, we have
 	\begin{eqnarray*}
 		\Vert \Phi(g)(t^{*})-\Phi(g)(0)\Vert_X & \leq &  \displaystyle\int_{0}^{t^{*}}\Vert R(t^{*},s)Bg(s)\Vert_X ds.
 	\end{eqnarray*}
 	Hence, $\lim\limits_{t^{*}\to 0}\Vert \Phi(g)(t^{*})-\Phi(g)(0)\Vert_X=0$ uniformly for $g\in \Omega$. As a consequence $\Phi(\Omega)(t)$ is equicontinuous for $t\geq 0$. Using Arzel\`{a}-Ascoli's Theorem, we get that $\Phi$ is compact.
 \end{proof}
 
 The following Theorem ensures the existence of an optimal solution for equation \eqref{5}.
 \begin{thm} \itshape Assume that map $R(t,s)$ is compact for $t-s>0$. Then, there is a unique pair $(x^*,u^*)\in \mathcal{A}_{ad}$ such that 
 	$$\min\{\mathcal{G}_\lambda(x,u):\hspace{0.1cm} (x,u)\in \mathcal{A}_{ad}\}=\mathcal{G}_\lambda(x^*,u^*),$$
 	where $x^*$ is the unique mild solution of equation $(\ref{5})$ corresponding to the control function $u^*$.
 \end{thm}
 \begin{proof} Let $$\mathcal{I}=\min\{\mathcal{G}_\lambda(x,u):\hspace{0.1cm} (x,u)\in \mathcal{A}_{ad}\}.$$ Then, there exists a sequence $(x_n,u_n)_{n\geq 0}\subseteq \mathcal{ A}_{ad}$ such that 
 	\begin{center}
 		$\lim\limits_{n\to +\infty}\mathcal{G}_\lambda(x_n,u_n)=\mathcal{I}$.
 	\end{center}
 	Hence, there is $r>0$ such that 
 	\begin{center}
 		$0\leq \mathcal{G}_\lambda(x_n,u_n)\leq r$ \text{ for } $n\geq 0$.
 	\end{center}  
 	In particular, there exists $C>0$ such that
 	\begin{equation}
 	\left(\displaystyle\int_{0}^{\tau}\Vert u_n(t)\Vert_U^2dt \right)^{1/2}\leq C \text{ for } n\geq 0.
 	\label{9}
 	\end{equation} 
 	Since,
 	\begin{equation*}
 	x_n(t)=R(t,0)x_0+\displaystyle\int_{0}^{t}R(t,s)Bu_n(s)ds \hspace{0.1cm} \text{ for } \hspace{0.1cm} n\geq 0,
 	\end{equation*}
 	it follows that for each $t\in[0,\tau]$, we have
 	\begin{eqnarray*}
 		\Vert x_n(t)\Vert_X &\leq & \Vert R(t,0)x_0\Vert_X +\displaystyle\int_{0}^{t}\Vert R(t,s)Bu_n(s)\Vert_{X} ds\\
 		&\leq& M_\tau\left[ \Vert x_0\Vert_X+M_B\sqrt{\tau}  \left( \displaystyle\int_{0}^{\tau}\Vert u_n(s)\Vert_U^2 ds \right)^{1/2} \right]\\
 		&\leq & M_\tau\left[ \Vert x_0\Vert_X+CM_B\sqrt{\tau}  \right]\\
 		&<& +\infty.
 	\end{eqnarray*}
 	Hence, 
 	\begin{center}
 		$\Vert x_n\Vert_{\mathbb{L}^{2}([0,\tau];X)}=\displaystyle\int_{0}^{\tau}\Vert x_n(t)\Vert_X^2dt<+\infty$.
 	\end{center}
 	From the reflexivity of $\mathbb{L}^{2}([0,\tau];X)$ ( since $X$ is reflexive ), we can find a subsequence $(x_{n_k})_{k\geq 0}$ of $(x_n)_{n\geq 0}$   such that
 	\begin{center}
 		$x_{n_k} \rightharpoonup x^*$ weakly  in   $ \mathbb{L}^{2}([0,\tau];X)$   as $  k\to +\infty$.
 	\end{center}
 	From (\ref{9}), we have $(u_n)_{n\geq 0}$ is a bounded sequence in $\mathbb{L}^{2}([0,\tau];U)$ which is a reflexive Banach space, then by the Banach-Alaoglu Theorem, we can extract a subsequence $(u_{n_k})_{k\geq 0}$ of $(u_n)_{n\geq 0}$ such that
 	\begin{center}
 		$u_{n_k} \rightharpoonup u^*$ weakly in  $\mathbb{L}^{2}([0,\tau];U)$  as $  k\to +\infty$.
 	\end{center} 
 	Using Lemma \ref{lemma 8}, we obtain that
 	\begin{eqnarray*}
 		& &\lim\limits_{k\to +\infty}\sup\limits_{t\in [0,\tau]}\Vert \displaystyle\int_{0}^{t}R(t,s)B[u_{n_k}(s)-u^*(s)]ds\Vert_X\\
 		&=&\lim\limits_{k\to +\infty}\sup\limits_{t\in [0,\tau]}\Vert \Phi(u_{n_k})(t)-\Phi(u^{*})(t)\Vert_X=0,
 	\end{eqnarray*}
 	which implies that 
 	\begin{equation}
 	\Vert x_{n_k}(t)-w^*(t)\Vert_{X}=\Vert \displaystyle\int_{0}^{t}R(t,s)B[u_{n_k}(s)-u^*(s)]ds\Vert_X \longrightarrow 0 \hspace{0.1cm} as \hspace{0.1cm}  k\to +\infty,
 	\label{12}
 	\end{equation}
 	where 
 	\begin{center}
 		$w^*(t)=R(t,0)x_0+\displaystyle\int_{0}^{t}R(t,s)Bu^*(s)ds$ \text{ for } $t\in [0,\tau]$,
 	\end{center}
 	is the unique mild solution of equation (\ref{5}) corresponding to the control $u^*$.
 	
 	Using the fact that the weak limit of $x_{n_k}$ is unique, it follows by (\ref{12}) that $w^*(t)=x^*(t)$ for each $t\in[0,\tau]$. Moreover,
 	\begin{center} 
 		$\lim\limits_{k\to +\infty}\sup\limits_{t\in [0,\tau]}\Vert x_{n_k}(t)-x^*(t)\Vert_{X}=\lim\limits_{k\to +\infty}\sup\limits_{t\in [0,\tau]}\Vert \displaystyle\int_{0}^{t}R(t,s)B(u_{n_k}(s)-u^*(s))ds\Vert_X=0$,
 	\end{center}
 	which implies that $x_{n_k}(\cdot)\to x^*(\cdot)$ strongly in $\mathcal{C}([0,\tau];X)$ as $k\to +\infty$. Since $x^*(\cdot)$ is a mild solution of equation (\ref{5}) with the control function $u^*$, then $(x^*,u^*)\in \mathcal{ A}_{ad}$. 
 	
 	We claim that $\mathcal{I}=\mathcal{G}_\lambda(x^*,u^*)$. Since $\mathcal{G}_\lambda(\cdot,\cdot)$ is convex and bounded, it follows by Proposition II.4.5 from \cite{9} that $\mathcal{G}_\lambda(\cdot,\cdot)$ is weakly lower semi-continuous. That is, for a sequence $(x_n,u_n)_{n\geq 0}$ weakly convergent to $ (x^*,u^*) $ in $\mathbb{L}^{2}([0,\tau];X)\times\mathbb{L}^{2}([0,\tau];U)$, we have
 	\begin{center}
 		$\mathcal{G}_\lambda(x^*,u^*)\leq\lim\limits_{n\to +\infty}\inf\mathcal{G}_\lambda(x_n,u_n)$.
 	\end{center}
 	Thus, 
 	\begin{center}
 		$ \mathcal{ I}\leq \mathcal{G}_\lambda(x^*,u^*)\leq \lim\limits_{n\to +\infty}\inf\mathcal{G}_\lambda(x_n,u_n)\leq \lim\limits_{n\to +\infty}\mathcal{G}_\lambda(x_n,u_n)=\mathcal{ I}$.
 	\end{center}
 	Consequently, $\mathcal{ I}=\mathcal{G}_\lambda(x^*,u^*)$. Using the fact that $\mathcal{G}_\lambda(\cdot,\cdot)$ is convex, we can affirm that $(x^*,u^*)$ is unique. 
 \end{proof}
 
 In the next, we provide an explicit form to the optimal control $u^*$. Firstly, we study the differentiability of the following map:
 \begin{center}
 	$\Psi:X\to \mathbb{R}^+$, $x\to \dfrac{1}{2}\Vert x\Vert^{2}_X$.
 \end{center}  
 It follows by \cite[Proposition 4.8]{11} and  \cite[Theorem 2.1]{10} that if $X^*$ is strictly convex (uniformly convex), then $\Psi$ is Gateaux differentiable (respectively, Fr\'{e}chet differentiable). In the both cases, we have
 \begin{center}
 	$\langle \partial_x\Psi(x),z\rangle={\dfrac{1}{2}\dfrac{d}{dx}\Vert x+\varepsilon z\Vert_X^2}_{\mid_{\varepsilon=0}}=\langle J(x),z\rangle$,
 \end{center}
 for $z\in X$, where $\partial_x\Psi(x)$ denotes the Gateaux (or Fr\'{e}chet) derivative of $\Psi$ at $x\in X$. That is the Gateaux (or Fr\'{e}chet) derivative of $\Psi$ is exactly the duality mapping $J$.
 \begin{lem} \itshape The optimal control $u^*$ is given  by 
 	\begin{center}
 		$	u^*(t)=B^*R(\tau,t)^*J(\mathcal{ R}(\lambda,\mathcal{B}_\tau)(d-R(\tau,0)x_0))$,
 	\end{center}
 	with $d\in X$.
 \end{lem}
 \begin{proof}
 	Let $(x^*,u^*)$ be the optimal solution of (\ref{5}). Then,
 	\begin{equation}
 	{\frac{d}{d\varepsilon}\mathcal{ G}(x_{u^*+\varepsilon v},u^*+\varepsilon v )}_{\mid_{\varepsilon=0}}=0,
 	\label{13}
 	\end{equation}
 	where $v\in\mathbb{L}^{2}([0,\tau];X)$ and $x_{u^*+\varepsilon v}$ is the mild solution of equation (\ref{5}) corresponding to the control function $u^*+\varepsilon v$. We recall that 
 	\begin{center}
 		$ x_{u^*+\varepsilon v}(t)=R(t,0)x_0+\displaystyle\int_{0}^{t}R(t,s)B[u^*(s)+\varepsilon v(s)]ds$ \text{ for } $t\in[0,\tau]$.
 	\end{center}
 	By (\ref{13}), we obtain that
 	\begin{eqnarray*}
 		0&=& {\frac{d}{d\varepsilon}\mathcal{ G}(x_{u^*+\varepsilon v},u^*+\varepsilon v )}_{\mid_{\varepsilon=0}}\\ 
 		&=& {\frac{d}{d\varepsilon}\left[ \Vert x_{u^*+\varepsilon v}(\tau)-d\Vert_{X}^2+\lambda \displaystyle\int_{0}^{\tau}\Vert u^*(t)+\varepsilon v(t) \Vert_U^2dt  \right]}_{\mid_{\varepsilon=0}}\\ 
 		&=& {2\left[ \langle J(x_{u^*+\varepsilon v}(\tau)-d), \frac{d}{d\varepsilon}(x_{u^*+\varepsilon v}(\tau)-d)\rangle\displaystyle\int_{0}^{\tau}(u^*(t)+\varepsilon v(t), \frac{d}{d\varepsilon}(u^*(t)+\varepsilon v(t)))dt \right]}_{\mid_{\varepsilon=0}}\\
 		&=& 2\langle J(x^*(\tau)-d), \displaystyle\int_{0}^{\tau}R(\tau,t)Bv(t)dt\rangle+2\lambda \displaystyle\int_{0}^{\tau}(u^*(t),v(t))dt.
 	\end{eqnarray*}
 	Hence, 
 	\begin{eqnarray*}
 		0&=&\displaystyle\int_{0}^{\tau}\langle J(x^*(\tau)-d),R(\tau,t)Bv(t)\rangle dt+\lambda \displaystyle\int_{0}^{\tau}(u^*(t),v(t))dt\\
 		&=& \displaystyle\int_{0}^{\tau}(B^*R(\tau,t)^*J(x^*(\tau)-d) +\lambda u^*(t),v(t))dt.
 	\end{eqnarray*}
 	We know that $v\in\mathbb{L}^{2}([0,\tau];U)$ is an arbitrary element, then we can chose $$v(t)=B^*R(\tau,t)^*J(x^*(\tau)-d) +\lambda u^*(t).$$ 
 	It follows that
 	\begin{equation}
 	u^*(t)=-\frac{1}{\lambda}B^*R(\tau,t)^*J(x^*(\tau)-d) \text{ for } t\in[0,\tau].
 	\label{14}
 	\end{equation}
 	$ u^*(\cdot)$ is continuous, moreover
 	\begin{eqnarray*}
 		x^{*}(\tau)&= &R(\tau,0)x_0-\displaystyle\int_{0}^{\tau}\frac{1}{\lambda} R(\tau,s)BB^*R(\tau,t)^*J(x^*(\tau)-d)ds\\ 
 		&=& R(\tau,0)x_0-\frac{1}{\lambda}\mathcal{ B}_{\tau}J(x^*(\tau)-d).
 	\end{eqnarray*}
 	Which implies that
 	\begin{center}
 		$x^{*}(\tau)-d=-\lambda\mathcal{R}(\lambda,\mathcal{B}_\tau)(d-R(\tau,0)x_0)$.
 	\end{center}
 	Combining the last equality with (\ref{14}), we get that 
 	\begin{center}
 		$	u^*(t)=B^*R(\tau,t)^*J(\mathcal{ R}(\lambda,\mathcal{B}_\tau)(d-R(\tau,0)x_0))$.
 	\end{center}
 \end{proof}
 %--------------------------------------------------------------------------------------------------------------
 \begin{thm} \itshape The following are equivalent:
 	\begin{enumerate}
 		\item[(i)]  Equation $(\ref{5})$ is approximately controllable on $[0,\tau]$.
 		\item[(ii)]  $\mathcal{B}_\tau$ satisfies $\langle x^*,\mathcal{B}_\tau x^*\rangle >0$ for each nonzero $x^*\in X^*$.
 		\item[(iii)] For every $\lambda>0$ and $x\in X$, we have the following
 		\begin{center}
 			$\Vert \lambda\mathcal{R}(\lambda,\mathcal{B}_\tau )x\Vert_X\leq \Vert x\Vert_X$ \text{ and } $\lim\limits_{\lambda\to  0^+}\Vert \lambda\mathcal{R}(\lambda,\mathcal{B}_\tau )x\Vert_X=0$.
 		\end{center}
 	\end{enumerate}
 	\label{thm 1}
 \end{thm}
 \begin{proof} $(i)\Rightarrow (ii)$. Assume that (\ref{5}) is approximately controllable on $[0,\tau]$. Then, $(\mathcal{L}_\tau)^{*}x^*=B^*R(\tau,t)x^*=0$, for all $ t\in[0,\tau]$ implies that $x^*=0$. Let $x^*\in X^*$ with $\Vert x^*\Vert_{X^*}\neq 0$, then
 	\begin{eqnarray*}
 		\langle x^*,\mathcal{B}_\tau x^*\rangle &=& \displaystyle\int_{0}^{\tau}\langle x^*,R(\tau,t)BB^*R(\tau,t)^*x^*\rangle dt\\ 
 		&=& \displaystyle\int_{0}^{\tau}(B^*R(\tau,t)^{*}x^*,B^*R(\tau,t)^{*}x^* )dt\\ 
 		&=& \displaystyle\int_{0}^{\tau}\Vert B^*R(\tau,t)^{*}x^*\Vert_{U}^2dt\\ 
 		&=& \Vert (\mathcal{L}_\tau)^{*}x^*\Vert_X^2>0.
 	\end{eqnarray*}
 	$(ii)\Rightarrow (iii)$. Assume that $(ii)$ holds. Let $x\in X$. Firstly, we prove that the following equation:
 	\begin{equation}
 	\lambda z+\mathcal{B}_\tau J(z)=\lambda x,
 	\label{6}
 	\end{equation}  has  a unique solution $\bar{z}\in X$. Since, the bijection of $J$, we transfer the equation (\ref{6}) in $X^*$. Let $z=J^{-1}(z^*)$, then we solve the equation $\lambda J^{-1}(z^*)+\mathcal{B}_\tau z^*=\lambda x$ in $X^*$. We define $\mathcal{A}:X^*\to X$ by 
 	\begin{center}
 		$\mathcal{A}(z^*)=\lambda J^{-1}(z^*)+\mathcal{B}_\tau(z^*)-\lambda x$ \text{ for } $z^*\in X^*$.
 	\end{center}
 	It is clear that $0\in Im(\mathcal{A})$ implies that there is $\bar{z}\in X$ such that $\lambda \bar{z}+\mathcal{B}_\tau J(\bar{z})=\lambda x$. We use the Minty-Browder Theorem \cite[Theorem 2.2]{20} to prove that $0\in Im(\mathcal{A})$. For that reason, it is sufficient to show that $\mathcal{A}$ has the following properties:
 	\begin{enumerate}
 		\item[(a)] $\mathcal{A}$ is strictly monotonic.
 		\item[(b)] $\mathcal{A}$ is demicontinuous.
 		\item[(c)] There is $\delta>0$ such that $\langle x^*,\mathcal{A} x^*\rangle>0$ for each $x^*\in X^*$ with $\Vert x^*\Vert_{X^{*}}>\delta$.
 	\end{enumerate}
 	In fact, let $x^*,y^*\in X^*$ with $x^*\neq y^*$. Then,
 	\begin{eqnarray*} 
 		\langle x^*-y^*,\mathcal{A} x^*-\mathcal{A} y^*\rangle
 		&=& \langle x^*-y^*,\lambda J^{-1}(x^*)+\mathcal{B}_\tau(x^*)-\lambda J^{-1}(y^*)-\mathcal{B}_\tau(y^*)\rangle\\ 
 		&=& \lambda\langle x^*-y^*, J^{-1}(x^*)- J^{-1}(y^*)\rangle+\langle x^*-y^*,\mathcal{B}_\tau(x^*-y^*)\rangle\\
 		&\geq&  \langle x^*-y^*,\mathcal{B}_\tau(x^*-y^*)\rangle\\
 		&>&0,
 	\end{eqnarray*}
 	which implies that condition (a) holds. Since $J^{-1}$ is demicontinuous and $\mathcal{B}_\tau$  is bounded, then $\mathcal{A}$ is demicontinuous and then condition (b) is satisfied. For (c), let $x^*\in X^*$, then
 	\begin{eqnarray*} 
 		\langle x^*,\mathcal{A} x^*\rangle&=& \langle x^*,\lambda J^{-1}(x^*)+\mathcal{B}_\tau(x^*)-\lambda x\rangle\\
 		&>& \lambda \langle x^*,J^{-1}(x^*)\rangle-\lambda \langle x^*,x\rangle\\
 		&\geq& \lambda(\Vert x^*\Vert_{X^*}^{2}-\Vert x\Vert_X \Vert x^*\Vert_{X^*})\\
 		&=&	\lambda \Vert x^*\Vert_{X^*}(\Vert x^*\Vert_{X^*}-\Vert x\Vert_{X}).
 	\end{eqnarray*}
 	Let $\delta=\Vert x\Vert_{X}$, then $\langle x^*,\mathcal{A} x^*\rangle>0$ for each $x^*\in X^*$ with $\Vert x^*\Vert_{X^*}>\delta$, which implies that condition (c) holds. Consequently,  the Minty-Browder Theorem shows that $0\in Im(\mathcal{A})$, that is there exists $\bar{z}\in X$ which solve equation (\ref{6}). Since, the strict monotonic of $\mathcal{A}$, $\bar{z}$ is unique.
 	
 	Now, we prove that $\Vert\lambda\mathcal{R}(\lambda,\mathcal{B}_\tau)x\Vert_X\leq \Vert x\Vert_X$. In fact, we have
 	\begin{eqnarray*}
 		\lambda \Vert \bar{z}\Vert_X^2=\lambda \langle \bar{z},J(\bar{z})\rangle &\leq& \lambda \langle \bar{z},J(\bar{z})\rangle+\langle \mathcal{B}_\tau J(\bar{z}) , J(\bar{z})\rangle\\
 		&=& \langle \lambda \bar{z}+\mathcal{B}_\tau J(\bar{z}),J(\bar{z})\rangle\\ 
 		&=&\langle \lambda x,J(\bar{z})\rangle \hspace{2cm} (**)\\
 		&\leq&\lambda\Vert x\Vert_X\Vert\bar{z}\Vert_X.
 	\end{eqnarray*}
 	Then, $\Vert\bar{z}\Vert_X=\Vert\lambda\mathcal{R}(\lambda,\mathcal{B}_\tau)x\Vert_X\leq \Vert x\Vert_X$. 
 	We prove that $\lim\limits_{\lambda\to  0^+}\Vert\lambda\mathcal{R}(\lambda,\mathcal{B}_\tau)x\Vert_X=0$ for each $x\in X$. Let $z_\lambda=\lambda\mathcal{R}(\lambda,\mathcal{B}_\tau)x$, since $\Vert z_\lambda \Vert_{X}\leq \Vert x\Vert_{X}$, we can extract a subsequence of $(z_\lambda)_{\lambda>0}$ that we continue to denote  by the same index $\lambda>0$ which is weakly convergent to some $z^*\in X^*$  that is $\langle J(z_\lambda), z\rangle\to \langle z^*,z\rangle$ as $\lambda\to 0^+$ for each $z\in X$. Using the fact that $J$ is bijective, we can find $z\in X$ such that $J(z)=z^*$. Then,
 	\begin{center}
 		$\lambda \langle z_\lambda,J(z)\rangle+\langle \mathcal{B}_\tau J(z),J(z)\rangle=\lambda \langle x,J(z)\rangle$.
 	\end{center}
 	Letting $\lambda\to 0^+$, we get that $\langle \mathcal{B}_\tau J(z),  J(z)\rangle=0$ and consequently, $J(z)=z^*=0$. From $(**)$, we can affirm that
 	\begin{center}
 		$\Vert z_\lambda\Vert_X\leq \langle  x,J(z_\lambda)\rangle$.
 	\end{center} 
 	Since $\langle x,J(z_\lambda)\rangle\to 0$ as $\lambda\to 0^+$, then
 	\begin{center}
 		$\Vert z_\lambda\Vert_X=\Vert \lambda\mathcal{R}(\lambda,\mathcal{B}_\tau)x\Vert_X\to 0^+ $ as $\lambda\to 0^+$.
 	\end{center}
 	$(iii)\Rightarrow (i)$. Assume that $(iii)$ holds. Let $d\in X$. The mild solution $x_\lambda(\cdot)$ of equation (\ref{5}) is given by 
 	\begin{center}
 		$x_\lambda(t)=R(t,0)x_0+\displaystyle\int_{0}^{t}R(t,s)Bu_\lambda(s)ds $ \text{ for } $t\in[0,\tau]$,
 	\end{center}
 	for 
 	\begin{center}
 		$u_\lambda(t)=B^*R(\tau,t)^*J(\mathcal{R}(\lambda,\mathcal{B}_\tau)(d-R(\tau,0)x_0))$,\quad $t\in[0,\tau]$.
 	\end{center}
 	Then,
 	\begin{eqnarray*}
 		x_\lambda(\tau)&=&R(\tau,0)x_0+\int_{0}^{\tau}R(\tau,s)BB^*R(\tau,s)^*J(\mathcal{R}(\lambda,\mathcal{B}_\tau)(d-R(\tau,0)x_0))ds\\ 
 		&=& R(\tau,0)x_0+\mathcal{B}_\tau J (\mathcal{R}(\lambda,\mathcal{B}_\tau)(d-R(\tau,0)x_0))\\
 		&=& d-\lambda \mathcal{R}(\lambda,\mathcal{B}_\tau)(d-R(\tau,0)x_0).
 	\end{eqnarray*}
 	Hence, 
 	\begin{center}
 		$\Vert x_\lambda(\tau)-d\Vert_X\leq \Vert \lambda \mathcal{R}(\lambda,\mathcal{B}_\tau)(d-R(\tau,0)x_0)\Vert_X $,
 	\end{center}
 	which implies that equation (\ref{5}) is approximately controllable on $[0,\tau]$.
 \end{proof}
 
 We assume the following assumptions.
 \begin{enumerate}
 	\item[$\bf\huge{(H_2})$] $U(t,s)$ is compact for $t-s>0$.
 	\item[$\bf\huge{(H_3})$] i) Let $x:]-\infty,\tau]\to X_\alpha(t_0)$ be such that $x_0=\phi$ and $x_{\mid_{[0,\tau]}}\in \mathcal{PC}_{\alpha}$. The function $t\to F(t,x_{\rho(t,x_t)})$ is strongly measurable on $[0,\tau]$, $t\to F(s,x_t)$ is continuous on $\Lambda(\rho^{-})\cup[0,\tau]$ for every $s\in[0,\tau]$ and the function $ F(t,\cdot):\mathcal{ H}_\alpha\to X$ is continuous.\\
 	ii)
 	For every $r>0$, there exists $\lambda_r(\cdot)\in \mathbb{L}^{\infty}([0,\tau];\mathbb{R}^+)$ such that 
 	\begin{center}
 		$\sup\limits_{\Vert \psi\Vert_{\mathcal{ H}_\alpha}\leq r}\Vert F(t,\psi)\Vert_X\leq \lambda_r(t)$ \text{ for a.e } $t\in[0,\tau]$,
 	\end{center}
 	and 
 	\begin{center}
 		$\liminf\limits_{r\to +\infty}\dfrac{\vert\lambda_r\vert_{\infty}}{r}=\bar{\delta}<+\infty$.
 	\end{center}
 \end{enumerate}
 
 Consider the following set:
 \begin{center}
 	$\mathcal{ Z}_\alpha=\left\{ x:]-\infty,\tau]\to X_\alpha(t_0), \hspace{0.1cm} x_0=0, \hspace{0.1cm} x_{\mid_{[0,\tau]}}\in \mathcal{C}_\alpha\right\}$,
 \end{center}
 with the norm $\Vert x\Vert_{\mathcal{ Z}_\alpha }=\sup\limits_{t\in[0,\tau]}\Vert x(t)\Vert_{\alpha,t_0}$. Then,  $(\mathcal{ Z}_\alpha,\Vert \cdot\Vert_{\mathcal{ Z}_\alpha })$ is a Banach space. Indeed, let $(x^n)_{n\geq 0}$ be a Cauchy sequence in $\mathcal{ Z}_\alpha$, that is for every $\varepsilon>0$ there is $N_0\in \mathbb{N}^{*}$ such that if $m>n>N_0$, we have
 \begin{center}
 	$\Vert x^n-x^m\Vert_{\mathcal{ Z}_\alpha }<\varepsilon$.
 \end{center}
 Since $x^n_0=0$ and $x^n_{\mid_{[0,\tau]}}\in \mathcal{C}_\alpha$, for each $n\geq 0$, then $(x^n)_{n\geq 0}$ is a Cauchy sequence in $\mathcal{C}_\alpha$, that is for every $\varepsilon>0$ there is $N_1\in \mathbb{N}^{*}$ ( may be $N_1=N_0$ ) such that if $m>n>N_1$, we have
 \begin{center}
 	$\Vert x^n-x^m\Vert_{\mathcal{C}_\alpha }<\varepsilon$.
 \end{center}
 Using the fact that $(\mathcal{C}_\alpha,\Vert\cdot\Vert_{\mathcal{C}_\alpha})$ is a Banach space, we can affirm that there exists $x\in \mathcal{C}_\alpha$ such that 
 \begin{center}
 	$\Vert x^n-x\Vert_{\mathcal{C}_\alpha}\to 0$ as $n\to +\infty$.
 \end{center}
 Let $z(\cdot)$ be the function defined from $]-\infty,\tau]$ to $X_\alpha(t_0)$ by 
 \begin{center}
 	$z(t)=\left\{\begin{array}{l}
 	x(t)$, \quad $t\in[0,\tau]\\
 	0$, \quad $t\in]-\infty ,0].
 	\end{array} \right.$
 \end{center}
 Then, $z_{\mid_{[0,\tau]}}\in\mathcal{C}_\alpha$, $z_0=0$ and $\lim\limits_{n\to +\infty}\Vert x^n-z\Vert_{\mathcal{ Z}_\alpha}=0$. Hence, $(\mathcal{ Z}_\alpha,\Vert \cdot\Vert_{\mathcal{ Z}_\alpha })$ is a Banach space.
 \section{Existence of the mild solution of equation (\ref{1})} \label{sec 4}
 \noindent
 
 Let $\mathcal{P}:\mathcal{ Z}_\alpha\to \mathcal{ Z}_\alpha $ such that  $\mathcal{P}(x)(t)=0$ for $t\in]-\infty,0]$ and
 \begin{equation*}
 \mathcal{P}(x)(t)=
 \displaystyle\int_0^t R(t,s)[Bu(s)+F(s,x_{\rho(s,x_s+y_s)}+y_{\rho(s,x_s+y_s)})]ds\hspace{0.1cm} \text{ for } \hspace{0.1cm} t\in[0,\tau],
 \end{equation*} 
 where, 
 \begin{center}
 	$y(t)=\left\{\begin{array}{l}
 	R(t,0)\phi(0)$, \quad $t\in[0,\tau]\\
 	\phi(t)$, \quad $t\in]-\infty ,0].
 	\end{array} \right.$
 \end{center}
 \begin{rem}
 	If $x$ is a fixed point of $\mathcal{ P}$, then $v=x+y$ is a mild solution of  equation (\ref{1}). Indeed, let $x$ be a fixed point of $\mathcal{ P}$. Then, $v(t)=\phi(t)$ for $t\leq 0$. For $t\in[0,\tau]$, we have 
 	\begin{eqnarray*}
 		v(t)&=&	R(t,0)\phi(0)+\displaystyle\int_0^t  R(t,s)[Bu(s)+F(s,x_{\rho(s,x_s+y_s)}+y_{\rho(s,x_s+y_s)})]ds \\
 		&=& R(t,0)\phi(0)+\displaystyle\int_0^t R(t,s)[Bu(s)+F(s,x_{\rho(s,v_s)}+y_{\rho(s,v_s)})]ds\\
 		&=& R(t,0)\phi(0)+\displaystyle\int_0^t R(t,s)[Bu(s)+F(s,v_{\rho(s,v_s)})]ds.
 	\end{eqnarray*} 
 	Hence, $v(\cdot)$ is a mild solution of equation (\ref{1}).
 	\label{rem 1}
 \end{rem}
 
 The following Theorem is the main result in this section.
 \begin{thm} \itshape Let $u(\cdot)\in\mathbb{L}^{\infty}([0,\tau];U)$. Assume that $\bf\huge{(H_0)}$-$\bf\huge{(H_3)}$ hold. If $$C_{\alpha,\beta}N_\beta H_3\bar{\delta}\dfrac{\tau^{1-\beta}}{1-\beta}<1,$$
 	then equation $(\ref{1})$ has a mild solution on $[0,\tau]$.
 	\label{thm 6}
 \end{thm}
 \begin{proof}
 	Let ${ E}_r=\{x\in\mathcal{ Z}_\alpha:\hspace{0.1cm} \Vert x\Vert_{\mathcal{ Z}_\alpha}\leq r \}$ for $r>0$. The proof is divided in three steps.
 	
 	\textbf{Step (1):} There is $r>0$ such that $\mathcal{ P}(E_r)\subseteq E_r$. By contradiction, assume that for every  $r>0$ there exist $t_r\in[0,\tau]$ and $x_r\in E_r$ such that $r<\Vert \mathcal{ P}(x_r)(t_r)\Vert_{\mathcal{ Z}_\alpha} $. Then,
 	\begin{eqnarray*}
 		r&<&\Vert \mathcal{ P}(x_r)(t_r)\Vert_{\alpha,t_0}\\ 
 		&=& \Vert \displaystyle\int_0^{t_r} R(t_r,s)[Bu(s)+F(s,x_{\rho(s,x_s+y_s)}+y_{\rho(s,x_s+y_s)})]ds \Vert_{\alpha,t_0}\\
 		&\leq&  \displaystyle\int_0^{t_r} \Vert R(t_r,s)[Bu(s)+F(s,x_{\rho(s,x_s+y_s)}+y_{\rho(s,x_s+y_s)})]\Vert_{\alpha,t_0}ds\\
 		&\leq& \displaystyle\int_0^{t_r} \Vert A^{\alpha}(t_0) R(t_r,s)[Bu(s)+F(s,x_{\rho(s,x_s+y_s)}+y_{\rho(s,x_s+y_s)})]\Vert_X ds\\
 		&\leq& \displaystyle\int_0^{t_r} \Vert A^{\alpha}(t_0) A^{-\beta}(t_r)A^{\beta}(t_r)  R(t_r,s)[Bu(s)+F(s,x_{\rho(s,x_s+y_s)}+y_{\rho(s,x_s+y_s)})]\Vert_X ds\\
 		&\leq & C_{\alpha, \beta} N_\beta \displaystyle\int_0^{t_r} \left[ \dfrac{\Vert Bu(s)\Vert_X}{(t_r-s)^{\beta}}+\dfrac{\Vert F(s,x_{\rho(s,x_s+y_s)}+y_{\rho(s,x_s+y_s)})  \Vert_X}{(t_r-s)^{\beta}}  \right]ds.
 	\end{eqnarray*}
 	Since, 
 	$$ \Vert x_{\rho(s,x_s+y_s)}+y_{\rho(s,x_s+y_s)} \Vert_{\mathcal{ H}_{\alpha}} \leq H_2\Vert \phi\Vert_{\mathcal{ H}_{\alpha}}+H_3r=r^*, $$
 	it follow that,
 	\begin{eqnarray*}
 		r&<&  C_{\alpha, \beta} N_\beta \displaystyle\int_0^{t_r} \left[ \dfrac{\Vert Bu(s)\Vert_X}{(t_r-s)^{\beta}}+\dfrac{\lambda_{r^*}(s)}{(t_r-s)^{\beta}}  \right]ds\\
 		&\leq & C_{\alpha, \beta} N_\beta \displaystyle\int_0^{t_r}  \dfrac{\Vert Bu(s)\Vert_X}{(t_r-s)^{\beta}} ds + C_{\alpha, \beta} N_\beta \vert \lambda_{r^*}\vert_{\infty}\displaystyle\int_{0}^{t_r} \dfrac{1}{(t_r-s)^{\beta}}ds\\
 		&\leq& C_{\alpha, \beta} N_\beta M_B \vert u\vert_{\infty}\dfrac{\tau^{\beta-1}}{1-\beta} + C_{\alpha, \beta} N_\beta\dfrac{\tau^{\beta-1}}{1-\beta}{\vert \lambda_{r^*}\vert_\infty}.
 	\end{eqnarray*}
 	Dividing by $r>0$, we get that 
 	\begin{center}
 		$1<\dfrac{C_{\alpha, \beta} N_\beta M_B \vert u\vert_{\infty}\dfrac{\tau^{\beta-1}}{1-\beta}}{r}+ C_{\alpha, \beta} N_\beta\dfrac{\tau^{\beta-1}}{1-\beta}\dfrac{\vert \lambda_{r^*}\vert_\infty}{r^*}\dfrac{r^*}{r}$.
 	\end{center}
 	Letting $r\to +\infty$, we obtain that
 	\begin{center}
 		$ 1\leq C_{\alpha,\beta}N_\beta H_3\bar{\delta}\dfrac{\tau^{1-\beta}}{1-\beta} $,
 	\end{center}
 	which is a contradiction. Hence, there exists $r>0$ such that $ \mathcal{P}(E_r)\subseteq E_r.$ 
 	
 	\textbf{Step (2):} The map $\mathcal{ P}$  is continuous. In fact, let $(x^n)_{n\geq 1}\subseteq \mathcal{ Z}_{\alpha}$ be such that $\lim\limits_{n\to +\infty}\Vert x^n-x\Vert_{\mathcal{ Z}_{\alpha}}=0$ for some $x\in \mathcal{ Z}_{\alpha}$. One side, we have
 	\begin{eqnarray*}
 		\mathcal{P}(x^n)(t)-\mathcal{P}(x)(t)
 		&=&\!\!\!\!\! \displaystyle\int_{0}^{t}\!\!\! R(t,s)[F(s,y_{\rho(s,x^n_s+y_s)}\!+\!x^n_{\rho(s,x^n_s+y_s)})\! -\!  F(s,y_{\rho(s,x_s+y_s)}\!+\!x_{\rho(s,x_s+y_s)})]ds.
 	\end{eqnarray*}
 	By Lemma \ref{lem 2.2}, we have 
 	\begin{eqnarray*}
 		\Vert x^n_s-x_s\Vert_{\mathcal{ H}_{\alpha}} &\leq& H_3\sup\limits_{t\in[0,\tau]}\Vert x^n(t)-x(t)\Vert_{\alpha,t_0}\\
 		&=& H_3\Vert x^n-x\Vert_{\mathcal{Z}_\alpha},
 	\end{eqnarray*}
 	which implies that $\lim\limits_{n\to +\infty}\Vert x^n_s-x_s\Vert_{\mathcal{ H}_{\alpha}}=0$ for each $s\in \Lambda(\rho^{-})\cup [0,\tau].$ Since the continuity of $\rho$, we get that 
 	\begin{center}
 		$\lim\limits_{n\to+\infty}	\Vert y_{\rho(s,x^n_s+y_s)}- y_{\rho(s,x_s+y_s)} +x^n_{\rho(s,x^n_s+y_s)}-x_{\rho(s,x_s+y_s)}\Vert_{\mathcal{ H_{\alpha}}}=0$, \text{ for a.e on } $[0,\tau]$. 
 	\end{center}
 	Using the fact that $F$ is continuous with respect to the second argument, we get that
 	\begin{center}
 		$ \lim\limits_{n\to+\infty} \Vert F(s,y_{\rho(s,x^n_s+y_s)}+x^n_{\rho(s,x^n_s+y_s)}) -F(s,y_{\rho(s,x_s+y_s)}+x_{\rho(s,x_s+y_s)})  \Vert_X=0 $,
 	\end{center}
 	for almost everywhere on $[0,\tau]$.  On other hand, we have
 	\begin{eqnarray*}
 		& &\Vert \mathcal{ P}(x^n)(t)-\mathcal{ P}(x)(t)\Vert_{\alpha,t_0}\\
 		&\leq& \displaystyle\int_{0}^{t} \Vert A^{\alpha}(t_0)R(t,s)[F(s,y_{\rho(s,x^n_s+y_s)}+x^n_{\rho(s,x^n_s+y_s)}) - F(s,y_{\rho(s,x_s+y_s)}+x_{\rho(s,x_s+y_s)}) ]\Vert_X ds\\ 
 		&\leq & C_{\alpha,\beta}N_\beta \displaystyle\int_{0}^{t}\dfrac{1}{(t-s)^{\beta}} \Vert F(s,y_{\rho(s,x^n_s+y_s)}+x^n_{\rho(s,x^n_s+y_s)}) - F(s,y_{\rho(s,x_s+y_s)}+x_{\rho(s,x_s+y_s)}) \Vert_X ds.
 	\end{eqnarray*}
 	Lebesgue dominate convergence Theorem implies that
 	\begin{center}
 		$\lim\limits_{n\to +\infty}\Vert \mathcal{P}(x^n)-\mathcal{P}(x)\Vert_{\mathcal{Z}_\alpha}=0$.
 	\end{center}
 	
 	\textbf{Step (3):} We use Arzela-Ascoli's Theorem to show that map $\mathcal{P}$ is compact. Then, we show that $\mathcal{P}(E_r)(t)$ is relatively compact for every $t\in[0,\tau]$ and $\mathcal{P}(E_r)$ is equicontinuous.
 	
 	Firstly, we prove that $V(t)=\mathcal{ P}(E_r)(t)$ is relatively compact for every $t\in[0,\tau]$. If $t=0$, $V(0)=\{  0\}$ is relatively compact in $X_\alpha(t_0)$. For $0<\alpha<\alpha^{'}<1$, the embedding $D(A^{\alpha ^{'}}(t_0))\hookrightarrow D(A^{\alpha}(t_0))$ is compact. Let $\alpha^{''}=\frac{\alpha^{'}+1}{2}$, and $t\in]0,\tau]$, then
 	\begin{eqnarray*}
 		\Vert A^{\alpha^{'}}(t_0)\mathcal{ P}(x)(t)\Vert_X
 		&\leq&  \displaystyle\int_{0}^{t}\Vert A^{\alpha^{'}}(t_0)R(t,s)[Bu(s)+F(s,x_{\rho(s,x_s+y_s)}+y_{\rho(s,x_s+y_s)})]\Vert_X  ds\\  
 		&\leq & \displaystyle\int_{0}^{t}\Vert A^{\alpha^{'}}(t_0)A^{-\alpha^{''}}(t)A^{\alpha^{''}}(t)R(t,s)[Bu(s) + F(s,x_{\rho(s,x_s+y_s)}+y_{\rho(s,x_s+y_s)})]\Vert_X  ds\\ 
 		&\leq & C_{\alpha^{''},\alpha^{'}}N_{\alpha^{''}} (M_B\vert u\vert_{\infty}+ \vert\lambda_r\vert_{\infty}) \frac{\tau^{1-\alpha^{''}}}{1-\alpha^{''}}.
 	\end{eqnarray*}
 	Then, $\{ A^{\alpha^{'}}(t_0)V(t)\}$ is bounded in $X$. Using the fact $A^{\alpha-\alpha^{'}}(t_0):X\to X_\alpha(t_0)$ is compact (see \cite{28}), we show that $V(t)$ is relatively compact for every $t\in[0,\tau]$.
 	
 	In the next, we prove that $\mathcal{P}(E_r)$ is equicontinuous.\\
 	\textbf{ Case 1.} Let $0<t'<t^{*}\leq \tau$. Then,
 	\begin{eqnarray*}
 		\mathcal{ P}(x)(t^{*})-\mathcal{ P}(x)(t')
 		&=&  \displaystyle\int_{0}^{t^{*}}R(t^{*},s)[Bu(s)+F(s,x_{\rho(s,x_s+y_s)}+y_{\rho(s,x_s+y_s)})]ds\\
 		&  & - \hspace{0.1cm} \displaystyle\int_{0}^{t'}R(t',s)[Bu(s)+F(s,x_{\rho(s,x_s+y_s)}+y_{\rho(s,x_s+y_s)})]ds\\
 		&=& \displaystyle\int_{0}^{t'}[R(t^{*},s)-R(t',s)][Bu(s)+F(s,x_{\rho(s,x_s+y_s)}+y_{\rho(s,x_s+y_s)})]ds\\
 		&  & + \hspace{0.1cm} \displaystyle\int_{t'}^{t^{*}}R(t^{*},s)[Bu(s)+F(s,x_{\rho(s,x_s+y_s)}+y_{\rho(s,x_s+y_s)})]ds.
 	\end{eqnarray*}
 	Hence, 
 	\begin{eqnarray*}
 		 \Vert  \mathcal{P}(x)(t^{*})- \mathcal{P}(x)(t')\Vert_{\alpha,t_0}
 		&\leq & \Vert\displaystyle\int_{0}^{t'}  [R(t^{*},s)-R(t',s)][Bu(s)+F(s,x_{\rho(s,x_s+y_s)}+y_{\rho(s,x_s+y_s)})] ds\Vert_{\alpha,t_0}\\
 		&  & + \hspace{0.1cm} \displaystyle\int_{t'}^{t^{*}}\Vert R(t^{*},s)[Bu(s)+F(s,x_{\rho(s,x_s+y_s)}+y_{\rho(s,x_s+y_s)})]\Vert_{\alpha,t_0} ds.
 	\end{eqnarray*}
 	It follows by $\bf\huge(H_0)$ and $\bf\huge(H_3)$ that 
 	\begin{eqnarray*}
 		\Vert  \mathcal{P}(x)(t^{*})- \mathcal{P}(x)(t')\Vert_{\alpha,t_0}
 		&\leq & \Vert A^{\alpha}(t_0) J(t',t^{*})\Vert_X
 		+ \dfrac{C_{\alpha,\beta} N_\beta [M_B \vert u\vert_\infty+\vert\lambda_{r^*}\vert_{\infty}]}{1-\beta} (t^{*}-t')^{1-\beta},
 	\end{eqnarray*}
 	where
 	\begin{center}
 		$J(t',t^{*})= \displaystyle\int_{0}^{t'}  [R(t^{*},s)-R(t',s)][Bu(s)+F(s,x_{\rho(s,x_s+y_s)}+y_{\rho(s,x_s+y_s)})] ds$.
 	\end{center}
 	Since,
 	\begin{eqnarray*}
 		&&\Vert A^{\alpha}(t_0) J(t',t^{*})\Vert_X\\
 		&\leq& \Vert A^{\alpha}(t_0) \displaystyle\int_{0}^{t'}  [R(t^{*},s)- R(t^{*},t')R(t',s)+R(t^{*},t')R(t',s)-R(t',s)]\\ & & \times [Bu(s)+F(s,x_{\rho(s,x_s+y_s)}+y_{\rho(s,x_s+y_s)})] ds\Vert_X\\
 		&\leq & \displaystyle\int_{0}^{t'}\Vert A^{\alpha}(t_0) [R(t^{*},s)- R(t^{*},t')R(t',s)] [Bu(s)+F(s,x_{\rho(s,x_s+y_s)}+y_{\rho(s,x_s+y_s)})]\Vert_X ds\\
 		&  & + \hspace{0.1cm} \Vert A^{\alpha}(t_0)[R(t^{*},t')-I]\displaystyle\int_{0}^{t'} R(t',s)[Bu(s)+F(s,x_{\rho(s,x_s+y_s)}+y_{\rho(s,x_s+y_s)})]ds\Vert_X.
 	\end{eqnarray*}
 	By Theorem \ref{thm 5}, we get that
 	\begin{eqnarray*}
 		&&\Vert  \mathcal{P}(x)(t^{*})- \mathcal{P}(x)(t')\Vert_{\alpha,t_0} \\
 		&\leq&  \left(\bar{C}_{\alpha,\beta}(t^{*}-t') [M_B \vert u\vert_\infty+\vert\lambda_{r}\vert_{\infty}]t'  + \dfrac{C_{\alpha,\beta} N_\beta [M_B \vert u\vert_\infty+\vert\lambda_{r^*}\vert_{\infty}]}{1-\beta}\right) (t^{*}-t')^{1-\beta} \\
 		&  & + \hspace{0.1cm}  \Vert [R(t^{*},t')-I]A^{\alpha}(t_0)\displaystyle\int_{0}^{t'} R(t',s)[Bu(s)+F(s,x_{\rho(s,x_s+y_s)}+y_{\rho(s,x_s+y_s)})]ds\Vert_X.
 	\end{eqnarray*}
 	Thanks to the compactness of 
 	\begin{center}
 		$\left\{\displaystyle\int_{0}^{t'} R(t',s)[Bu(s)+F(s,x_{\rho(s,x_s+y_s)}+y_{\rho(s,x_s+y_s)})]ds:\hspace{0.1cm} x\in E_r \right\}$
 	\end{center}
 	in $X_\alpha(t_0)$ and assumption $\bf\huge(H_2)$, we obtain that
 	\begin{center}
 		$\lim\limits_{ t^{*}\to t'}\Vert  \mathcal{P}(x)(t^{*})- \mathcal{P}(x)(t')\Vert_{\alpha,t_0}=0$ uniformly for $x\in E_r$.
 	\end{center}
 	\textbf{ Case 2.}
 	If $t'=0$, let $0<t^{*}<t_1$, we get that 
 	\begin{center}
 		$\Vert  \mathcal{P}(x)(t^{*})- \mathcal{P}(x)(0)\Vert_{\alpha,t_0}\leq \dfrac{C_{\alpha,\beta} N_\beta [M_B \vert u\vert_\infty+\vert\lambda_{r^*}\vert_{\infty}]}{1-\beta} {t^{*}}^{1-\beta}$.
 	\end{center}
 	Hence 
 	\begin{center}
 		$\lim\limits_{ t^{*}\to 0}\Vert  \mathcal{P}(x)(t^{*})- \mathcal{P}(x)(0)\Vert_{\alpha,t_0}=0$ uniformly for $x\in E_r$.
 	\end{center}
 	By Arzel\`a-Ascoli's Theorem, we conclude that $\mathcal{ P}$ is compact. By Schauder's Fixed Point Theorem, we conclude that $\mathcal{ P}$ has a fixed point. 
 \end{proof}
 
 In the next, we assume the following assumptions.
 \begin{enumerate}
 	\item[$\bf\huge(H_4)$]  Let $x:]-\infty,\tau]\to X_\alpha(t_0)$ be such that $x_0=\phi$ and $x_{\mid_{[0,\tau]}}\in \mathcal{PC}_{\alpha}$. The function $t\to F(t,x_{\rho(t,x_t)})$ is strongly measurable on $[0,\tau]$, $t\to F(s,x_t)$ is continuous on $\Lambda(\rho^{-})\cup[0,\tau]$ for every $s\in[0,\tau]$ and the function $ F(t,\cdot):\mathcal{ H}_\alpha\to X$ is continuous. Moreover, there exists $N>0$ such that 
 	\begin{center}
 		$\Vert F(t,\psi)\Vert_X\leq N$ \text{ for } $\psi\in \mathcal{ H_{\alpha}}$ \text{ and a.e } $t\in[0,\tau]$.
 	\end{center}
 	
 	\item[$\bf\huge(H_5)$] The linear equation (\ref{5}) is approximately controllable on $[0,\tau]$.
 \end{enumerate}
\section{Approximate controllability for equation (\ref{1})} \label{sec 5}
\noindent

To prove that equation (\ref{1}) is approximately controllable, we need to show that for each $\lambda>0$ and $d\in X$, equation (\ref{1}) has a mild solution using the following control function:
\begin{center}
	$u_\lambda(t)=B^{*}R(\tau,t)^{*}J(\mathcal{ R}(\lambda,\mathcal{ B}_\tau)p(x^\lambda(\cdot)))$,
\end{center}
where 
\begin{center}
	$p(x^\lambda(\cdot))= d-R(\tau,0)\phi(0) -\displaystyle\int_{0}^{\tau}R(\tau,s)F(s,x^{\lambda}_{\rho(s,x^{\lambda}_s)})ds$,
\end{center}
and $x^\lambda:]-\infty,\tau]\to X_\alpha(t_0)$ such that $x_0^\lambda=\phi$ and $x^\lambda_{\mid_{[0,\tau]}}\in \mathcal{C}_\alpha$.

\begin{thm} \itshape Assume that $\bf\huge(H_0)$-$\bf\huge(H_2)$, $\bf\huge(H_4)$, $\bf\huge(H_5)$ hold. Then, equation $(\ref{1})$ is approximately controllable on $[0,\tau]$.
	\label{thm 2}
\end{thm}
\begin{proof}
	Let $\mathcal{ P}_\lambda$ the map defined on $\mathcal{ Z}_\alpha$ by $\mathcal{P}_\lambda(x)(t)=0$ if $t\in]-\infty,0]$ and
	\begin{equation*}
	\mathcal{P}_\lambda(x)(t)=
	\displaystyle\int_0^t R(t,s)[Bu_\lambda(s)+F(s,x_{\rho(s,x_s+y_s)}+y_{\rho(s,x_s+y_s)})]ds \text{ for } t\in[0,\tau]
	\end{equation*} 
	where, 
	\begin{center}
		$y(t)=\left\{\begin{array}{l}
		R(t,0)\phi(0)$, \quad for $t\in[0,\tau]\\
		\phi(t)$, \quad for $t\in]-\infty ,0].
		\end{array} \right.$
	\end{center}
	
	It follows by Remark \ref{rem 1}, that if $x^{\lambda}(\cdot)$ is a fixed point of $\mathcal{ P}_\lambda$ then, $x^{\lambda}(\cdot)+y(\cdot)$ is a mild solution of equation (\ref{1}) with the control function $u_\lambda(\cdot)$.
	
	Since, 
	\begin{eqnarray*}
		\Vert u_\lambda(t)\Vert_U &=&\Vert B^{*}R(\tau,t)^{*}J(\mathcal{ R}(\lambda,\mathcal{ B}_\tau)p(x(\cdot))) \Vert_U\\
		&\leq & M_BM_\tau \Vert \mathcal{ R}(\lambda,\mathcal{ B}_\tau)p(x(\cdot))\Vert_X\\
		&\leq & \dfrac{M_BM_\tau}{\lambda}\Vert p(x(\cdot))\Vert_X\\
		&\leq & \dfrac{M_BM_\tau}{\lambda}\left[ \Vert d\Vert_X+M_\tau\left(\Vert  \phi(0)\Vert_X+ N\tau \right)  \right]\\
		&<& +\infty,
	\end{eqnarray*}
	for each $x\in \mathcal{ Z}_\alpha$, then $u_\lambda(\cdot)\in \mathbb{L}^{\infty}([0,\tau];U)$ for each $\lambda>0$. It follows by assumption $\bf\huge{(H_4)}$, that $\bar{\delta}=0$. According to the proof of Theorem \ref{thm 6}, we can show that for each  $\lambda>0$, there exists  $r>0$ such that $\mathcal{ P}_\lambda(E_r)\subseteq E_r$, where $E_r$ is the same as defined in the proof of Theorem \ref{thm 6}.
	
	Next, we prove that $\mathcal{ P}_\lambda$ is continuous on $\mathcal{ Z}_\alpha$. Let $(x^n)_{n\geq 0}\subseteq \mathcal{ Z}_\alpha$ be such that $\lim\limits_{n\to +\infty}\Vert x^n-x\Vert_{\mathcal{ Z}_\alpha}=0$ for some $x\in \mathcal{ Z}_\alpha$. Since,  
	\begin{eqnarray*}
		&&\Vert p(x^n(\cdot))-p(x(\cdot))\Vert_X\\
		&\leq&  \Vert \displaystyle\int_{0}^{\tau} R(\tau,s)[F(s,x^n_{\rho(s,x^n_s+y_s)}+y_{\rho(s,x^n_s+y_s)}) -F(s,x_{\rho(s,x_s+y_s)}+y_{\rho(s,x_s+y_s)})] ds\Vert_X\\
		&\leq & M_\tau \displaystyle\int_{0}^{\tau}\Vert F(s,x^n_{\rho(s,x^n_s+y_s)}+y_{\rho(s,x^n_s+y_s)}) -F(s,x_{\rho(s,x_s+y_s)}+y_{\rho(s,x_s+y_s)}) \Vert_X ds \\
		&\leq & M_\tau \displaystyle\int_{0}^{\tau}\Vert F(s,x^n_{\rho(s,x^n_s+y_s)}+y_{\rho(s,x^n_s+y_s)}) -F(s,x_{\rho(s,x_s+y_s)}+y_{\rho(s,x_s+y_s)}) \Vert_X ds.
	\end{eqnarray*}
	Using the fact that $F(t,\cdot)$ is continuous, we deduce that $$\lim\limits_{n\to +\infty}\Vert p(x^n(\cdot))-p(x(\cdot))\Vert_X=0.$$ 
	Similarly, as in the proof of Theorem \ref{thm 6}, we show that $P_\lambda(\cdot)$ is continuous for each $\lambda>0$. As in step tree of the proof of Theorem \ref{thm 6}, we can prove that $\mathcal{ P}_\lambda$ is compact for each $\lambda>0$.  Consequently, the Schauder's Fixed Point Theorem shows that for each $\lambda>0$, equation (\ref{1}) has a mild solution $x_\lambda$ corresponding to the control function $u_\lambda(\cdot)$. Finally, we prove that equation (\ref{1}) is approximately controllable on $[0,\tau]$. Let $(x^{\lambda})_{\lambda>0}$ be a sequence of mild solutions of equation (\ref{1}) with the following sequence of control functions:
	\begin{center}
		$u_\lambda(t)=B^{*}R(\tau,t)^{*}J(\mathcal{ R}(\lambda,\mathcal{ B}_\tau)p(x^\lambda(\cdot)))$.
	\end{center}
	It is easy seen that 
	\begin{center}
		$x^{\lambda}(\tau)=d-\lambda \mathcal{ R}(\lambda,\mathcal{ B}_\tau)p(x^\lambda(\cdot))$. 
	\end{center}
	By assumption $\bf\huge(H_4)$, we get that 
	\begin{center}
		$\displaystyle\int_{0}^{\tau}\Vert F(s,x^{\lambda}_{\rho(s,x_s^\lambda)})\Vert_X^{2} ds\leq N^{2}\tau$,
	\end{center}
	which implies that $(F(s,x^{\lambda}_{\rho(s,x_s^\lambda)}))_{\lambda>0}$ is a bounded sequence in $\mathbb{L}^{2}([0,\tau];X)$ which is a reflexive Banach space (because of the reflexivity of $X$). By Banach-Alaoglu Theorem, we may extract a subsequence of $(F(s,x^{\lambda}_{\rho(s,x_s^\lambda)}))_{\lambda>0}$  that we continue to denote by the same index $\lambda>0$ such that 
	\begin{center}
		$\{s\to F(s,x^{\lambda}_{\rho(s,x_s^\lambda)}):\hspace{0.1cm} \lambda>0 \}\rightharpoonup F(\cdot)$  weakly in $\mathbb{L}^{2}([0,\tau];X)$, 
	\end{center}
	for some $F(\cdot)\in \mathbb{L}^{2}([0,\tau];X) $ as $\lambda\to 0^{+}.$ 
	
	Let 
	\begin{center}
		$	w^*=d-R(\tau,0)\phi(0)-\displaystyle\int_{0}^{\tau}R(\tau,s)F(s)ds $.
	\end{center}
	Since the compactness of $R(t,s)$ for $t-s>0$, we obtain
	\begin{eqnarray*}
		\Vert p(x^\lambda)-w^*\Vert_X &\leq & \Vert \displaystyle\int_{0}^{\tau}R(\tau,s)[F(s,x^\lambda_{\rho(s,x^\lambda_s)})-F(s)]ds\Vert_X\\
		& & \longrightarrow 0 \hspace{0.5cm} as \hspace{0.5cm} \lambda\to 0^+.
	\end{eqnarray*}
	
	On other hand, we have
	\begin{center}
		$\Vert x^\lambda(\tau)-d\Vert_X \leq\Vert \lambda \mathcal{  R}(\tau,\mathcal{ B}_\tau)w^*\Vert_X+\Vert p(x^\lambda)-w^*\Vert_X$.
	\end{center}
	Combining assumption $\bf\huge(H_5)$ with Theorem \ref{thm 1}, we can affirm that 
	\begin{center}
		$\Vert x^\lambda(\tau)-d\Vert_X\to 0$ as $\lambda\to 0^+$.
	\end{center}
	Consequently, equation (\ref{1}) is approximately controllable on $[0,\tau]$.
\end{proof}
\section{Application} \label{sec 6}
\noindent

To illustrate our basic results, we propose the following one-dimensional delayed heat conduction system, which has numerous physical applications, specifically in the theory of heat conduction in materials with memory (see \cite{Pruss,Pandolfi}). The system is represented by the following partial integrodifferential equation:
\begin{equation}
\left\{\hspace{-0.2cm} \begin{array}{l}
\dfrac{\partial y(t,\xi)}{\partial t} =\dfrac{\partial^{2} y(t,\xi)}{\partial {\xi}^{2}} + b(t)y(t,\xi)\!+\!\displaystyle\int_{0}^{t} C(t,s) \dfrac{\partial^{2} y(s,\xi)}{\partial {\xi}^{2}} ds+ \chi_{]a_1,a_2[}(\xi)\upsilon(t,\xi) \\ \\  \hspace{0.8cm}+
h\left[t,y(t-\sigma(\vert y(t,\xi)\vert),\xi), \dfrac{\partial y(t-\sigma(\vert y(t,\xi)\vert ),\xi)}{\partial {\xi}}  \right]\!, t\in[0,\tau],\hspace{0.1cm}\xi\in[0,\pi]\\ \\
y(t,0)=y(t,\pi)=0, \quad t\in[0,\tau],\\ \\
y(t,\xi)=\phi(t)(\xi), \quad (t,\xi)\in]-\infty,0]\times[0,\pi],
\end{array}\right.
\label{17}
\end{equation}
where $b(t)$ is H\" older continuous with order $0<\kappa < 1$, that is there exists a positive constant $C_b$ such that 
\begin{center}
	$\vert b(t)-b(s)\vert \leq C_b \vert t-s\vert^{k}$ \text{ for } $t,s\in\mathbb{R}^{+}$.
\end{center}
Moreover, $b(\cdot)$ is continuously differentiable and $b(t)<-1$.  The function $C(\cdot,\cdot)\in BU(\mathbb{R}^{+}\times\mathbb{R}^{+},\mathbb{R}^{*}),$ where $BU(\mathbb{R}^{+}\times\mathbb{R}^{+},\mathbb{R}^{*})$ is the space of all defined bounded uniformly continuous functions from $\mathbb{R}^{+}\times\mathbb{R}^{+}$ to $\mathbb{R}^{*}$. Moreover, $t\to \frac{\partial }{\partial t}C(t,s)$ is bounded and continuous from $\mathbb{R}^{+}$ to $\mathbb{R}$. Function $\upsilon\in\mathbb{L}^{2}([0,\tau]\times (0,\pi);\mathbb{R})$. Let $X=U=\mathbb{L}^{2}((0,\pi);\mathbb{R})$. Here, $(A(t))_{t\geq 0}$ and $\{G(t,s):\hspace{0.1cm} 0\leq s\leq t \}$ are given by:
\begin{equation*}
\left\{\begin{array}{l}
(A(t)f)(\xi)=-f^{''}(\xi)-b(t)f(\xi), \quad t\geq 0, \quad \xi\in (0,\pi),\\
\text{ for }	f\in D(A(t))=D(A)=H^{2}(0,\pi)\cap H^{1}_{0}(0,\pi),
\end{array}  \right.
\end{equation*}
and 
\begin{equation*}
\left\{\begin{array}{l}
(G(t,s)f)(\xi)=C(t,s)f^{''}(\xi), \quad t\geq s\geq 0, \quad \xi\in (0,\pi),\\ 
\text{ for }	f\in D(A(t))=D(A)=H^{2}(0,\pi)\cap H^{1}_{0}(0,\pi).
\end{array}  \right.
\end{equation*}

In the sequel, we show that assumptions $\bf\huge{(R_1)-(R_3)}$ and $\bf\huge{ (C_1)-(C_3)}$ are satisfied.\\
$\bullet$ $\bf\huge{Assumptions:\hspace{0.1cm} (R_1)-(R_3)}$.
\begin{enumerate}
	\item[] -  $Assumption$ $\bf\huge{(R_1)}$: for each $t\geq 0$, $-A(t)=\Delta +b(t)I$ ($I$ is the identity map) generates a strongly continuous semigroup on $X$. Moreover, $-A(t)y=\Delta y +b(t)y$ is strongly continuously differentiable on $\mathbb{R}^{+}$ for each $y\in D(A)$, due to the strong continuous differentiability of $b(\cdot)$. We known that $\Delta$ is an infinitesimal generator of a $C_0$-semigroup of contraction on $X$. Since $b(t)<-1$, it follows that $\Delta+b(t)I$ is an infinitesimal generator of a $C_0$-semigroup of contraction on $X$. Thus, $-A(t)=\Delta+b(t)I$ is stable \cite[page: 131]{28}. Furthermore, we have $(\tilde{G}(t)y)(\cdot)=C(t+\cdot,t)\Delta y$ for each $y\in D(A)$. Since $t\to \frac{\partial C(t,s) }{\partial t}$ is continuously differentiable from $\mathbb{R}^{+}$ to $\mathbb{R}$, then $\tilde{G}(t)y$ is strongly continuously differentiable on $\mathbb{R}^{+}$ for each $y\in Y$. Consequently, assumption $\bf\huge{(R_1)}$ holds.
	\item[] -  $Assumption$ $\bf\huge{(R_2)}$: since $C(\cdot,\cdot)$ is continuous, then $\tilde{G}(t)$ is continuous on $\mathbb{R}^{+}$. Moreover, 
	\begin{eqnarray*}
		\Vert \tilde{G}(t)y\Vert_{\mathcal{  F}}&=& \sup\limits_{s\in\mathbb{R}^+}\Vert (\tilde{G}(t)y)(s)\Vert_X\\
		&=& \sup\limits_{s\in\mathbb{R}^+}\vert C(t+s,t)\vert \Vert \Delta y\Vert_X\\
		&\leq & \sup\limits_{s\in\mathbb{R}^+}\vert C(t+s,t)\vert (\Vert \Delta y\Vert_X+\Vert y\Vert_X).
	\end{eqnarray*}
	Then, assumption $\bf\huge{(R_2)}$ holds.
	\item[] -  $Assumption$ $\bf\huge{(R_3)}$: let $y\in Y$, then $\dfrac{d}{ds} (\tilde{G}(t)y)(s)=\dfrac{\partial }{\partial s}C(t+s,t)\Delta y$, which implies that $\tilde{G}(t)y\in Dom(D)$ and 
	\begin{center}
		$(D\tilde{G}(t)y)(s)=\dfrac{\partial }{\partial s}C(t+s,t)\Delta y$.
	\end{center}
	Moreover, $t\to D\tilde{G}(t)$ is continuous on $\mathbb{R}^{+}$ and 
	\begin{center}
		$\Vert D\tilde{G}(t)y\Vert_{\mathcal{  F}}\leq \sup\limits_{s\in\mathbb{R}^+}\vert \dfrac{\partial}{\partial s}C(t+s,t)\vert(\Vert \Delta y\Vert_X+\Vert y\Vert_X)$.
	\end{center}
	Hence, $D\tilde{G}(t)\in\mathcal{  L}(Y,\mathcal{  F})$. Thus, assumption $\bf\huge{(R_3)}$ holds.
\end{enumerate}
$\bullet$ $\bf\huge{Assumptions:\hspace{0.1cm} (C_1)-(C_3)}$. 
\begin{enumerate}
	\item[-] The assumption $\bf\huge{(C_1)}$ follows from assumption $\bf\huge{(R_1)}$. The assumptions $\bf\huge{(C_2)}$ and $\bf\huge{(C_3)}$ hold  with $K_1=1$, $K_2=\frac{C_b}{3}$ and $\gamma=\kappa$. Hence, there exists a unique evolution system $\{U(t,s):\hspace{0.1cm} 0\leq s\leq t\}$ generated by the family $\{-A(t):\hspace{0.1cm} t\geq 0 \}$, which is given by 
	\begin{center}
		$ U(t,s)x=\sum\limits_{n=1}^{+\infty}\exp\left(-n^2(t-s)+\int_{s}^{t}b(r)dr\right)\langle x,e_n\rangle e_n$ \text{ for } $t\geq s\geq 0$ \text{ and } $x\in X$,
	\end{center}
	where $e_n(\xi)=\sqrt{\frac{2}{\pi}}\sin(n \xi)$ for $n\geq 1$ and $ \xi\in (0,\pi)$.
\end{enumerate}
\begin{prop}
	Let $0<\alpha\leq 1$. The fractional power $A^{\alpha}(t_0)$ for $t_0\in[0,\tau]$ is given by
	\begin{equation*}
	\left\{\begin{array}{l}
	D(A^{\alpha}(t_0))=\left\{w\in X: \hspace{0.1cm}\sum\limits_{n=1}^{+\infty}(n^{2}-b(t_0))^{\alpha}\langle w,e_n\rangle e_n \in X\right\},\\
	A^{\alpha}(t_0)w=\sum\limits_{n=1}^{+\infty}(n^{2}-b(t_0))^{\alpha}\langle w,e_n\rangle e_n \hspace{0.1cm}  \text{ for } \hspace{0.1cm} w\in D(A^{\alpha
	}(t_0)).
	\end{array}\right.
	\end{equation*}
\end{prop}
\begin{proof} Let $(T_{t_0}(s))_{s\geq 0}$ be the analytic $C_0$-semigroup generated by $-A(t_0)$. We recall that
	\begin{center}
		$T_{t_0}(s)w=\sum\limits_{n=1}^{+\infty}e^{-(n^{2}-b(t_0))s}\langle w,e_n\rangle e_n$ \text{ for } $w\in X$.
	\end{center}
	Let $w\in D(A^{\alpha}(t_0))$. Then,
	\begin{eqnarray*}
		A^{\alpha}(t_0) w &=& A(t_0)A^{-(1-\alpha)}(t_0) w\\
		&=& A(t_0)\left[\dfrac{1}{\Gamma(1-\alpha)}\displaystyle\int_{0}^{+\infty}s^{-\alpha}T_{t_0}(s)wds\right]\\
		&=& A(t_0)\left[\sum\limits_{n=1}^{+\infty}\dfrac{1}{\Gamma(1-\alpha)}\displaystyle\int_{0}^{+\infty}s^{-\alpha}e^{-(n^{2}-b(t_0))s}ds\langle w,e_n\rangle e_n\right]\\
		&=&\sum\limits_{n=1}^{+\infty}\left[\dfrac{1}{\Gamma(1-\alpha)}\displaystyle\int_{0}^{+\infty}(n^{2}-b(t_0))s^{-\alpha}e^{-(n^{2}-b(t_0))s}ds\right]\langle w,e_n\rangle e_n.
	\end{eqnarray*}
	Let $r=(n^{2}-b(t_0))s$, it follows that $s^{-\alpha}=(n^{2}-b(t_0))^{\alpha}r^{-\alpha}$. Hence,
	\begin{eqnarray*}
		A^{\alpha}(t_0) w &=& \sum\limits_{n=1}^{+\infty}(n^{2}-b(t_0))^{\alpha}\left[\dfrac{1}{\Gamma(1-\alpha)}\displaystyle\int_{0}^{+\infty}r^{-\alpha}e^{-r}dr\right]\langle w,e_n\rangle e_n.
	\end{eqnarray*}
	The result follows the fact that $\Gamma(1-\alpha)=\displaystyle\int_{0}^{+\infty}r^{-\alpha}e^{-r}dr$ (by the definition of Gamma function).
\end{proof}
\begin{lem} \cite{Webb}
	If $w\in D(A^{\frac{1}{2}}(t_0))$, then $w$ is absolutely continuous, $w^{'}\in X$ and $\Vert w^{'}\Vert_{X}=\Vert A^{\frac{1}{2}}(t_0)w\Vert_{X}$.
\end{lem}
%================================================================================================================

Let $\gamma>0$, and
\begin{center}
	$\mathcal{H}=\mathcal{  C}_\gamma=\left\{ \psi\in \mathcal{  C}(]-\infty,0];X): \hspace{0.1cm} \lim\limits_{\theta\to -\infty}e^{\gamma\theta}\psi(\theta) \text{ exists in } X\right\}$
\end{center} 
equipped with the following norm 
\begin{center}
	$\Vert\psi\Vert_{\mathcal{  H}}=\sup\limits_{\theta\leq 0}e^{\gamma\theta}\Vert\psi(\theta)\Vert_X$ \text{ for } $\psi\in\mathcal{H}$.
\end{center}
$\mathcal{  H}$ is a phase space which satisfies axioms $\bf\huge(A_1)$-$\bf\huge(A_3)$, see \cite[Theorem 3.7]{15}.

Assumption $\bf\huge(H_1)$ is satisfied. Let $\mathcal{  H}_{\frac{1}{2}}=\mathcal{  C}_{\gamma,\frac{1}{2}}$, where 
\begin{center}
	$\mathcal{  C}_{\gamma,\frac{1}{2}}=\left\{ \psi\in \mathcal{H}:\hspace{0.1cm} \psi(\theta)\in X_{\frac{1}{2}}(t_0) \text{ for } \theta\leq 0 \text{ and } A^{\frac{1}{2}}(t_0)\psi\in \mathcal{H} \right\}$,
\end{center}
equipped with the following norm
\begin{center}
	$\Vert\psi\Vert_{\mathcal{  H}_{\frac{1}{2}}}=\sup\limits_{\theta\leq 0}e^{\gamma\theta}\Vert  A^{\frac{1}{2}}(t_0)\psi(\theta)\Vert_X$ \text{ for } $\psi\in \mathcal{  H}_{\frac{1}{2}}$.
\end{center}
By Lemma \ref{Lemma 6}, we obtain that $\mathcal{  H}_{\frac{1}{2}}$  is a phase space that satisfies axioms $\bf\huge(A'_1)$-$\bf\huge(A'_3)$.
%================================================================================================================

Now, we define the function  $x:[0,\tau]\to X$ by
\begin{center}
	$x(t)(\xi)=y(t,\xi)$ \text{ for } $t\in[0,\tau]$ and $\xi\in[0,\pi]$,
\end{center}
the map $B:X\to X$ by
\begin{center}
	$(Bu)(\xi)=\chi_{]a_1,a_2[}(\xi)u(\xi)$ \text{ for } $\xi\in[0,\pi]$,
\end{center}
the function $F:[0,\tau]\times \mathcal{  H}_{\frac{1}{2}}\to X$ by
\begin{center}
	$F(t,\psi)(\xi)=h(t,\psi(\xi),\psi^{'}(\xi))$ \text{ for } $t\in[0,\tau]$, $\xi\in[0,\pi]$ and $\psi\in\mathcal{  H}_{\frac{1}{2}} $, 
\end{center}
the control $u:[0,\tau]\to X$ by
\begin{center}
	$u(t)(\xi)=\upsilon(t,\xi)$ \text{ for } $t\in[0,\tau]$ and $\xi\in ]a_1,a_2[$,
\end{center}
and the state-dependent delay function $\rho:[0,\tau]\times\mathcal{  H}_{\frac{1}{2}}\to ]-\infty,\tau]$ by
\begin{center}
	$\rho(t,\psi)=t-\sigma(\Vert\psi(0)\Vert_X )$ \text{ for } $t\in[0,\tau]$ and $\psi\in\mathcal{  H}_{\frac{1}{2}} $,
\end{center}
where $\sigma(\cdot)$ is a  continuous function on $\mathbb{R}^{+}$ and $h:[0,\tau]\times \mathbb{R}\times\mathbb{R}$ is a continuous function satisfying
\begin{center}
	$\left(\displaystyle\int_{0}^{1}\vert h(t,\psi_1(\theta)(\xi),\psi_2(\theta)(\xi))\vert^{2} d\xi\right)^{1/2}\leq M_h$, \quad $\psi_1,\psi_2\in\mathcal{  H}_{\frac{1}{2}}$, \text{ for some } $M_h\geq 0$.
\end{center}

Then, equation (\ref{17}) takes the following form:
\begin{equation*}
\left\{\begin{array}{l}
x^{'}(t)=-A(t)x(t)+\displaystyle\int_{0}^{t}G(t,s)x(s)ds+F(t,x_{\rho(t,x_t)}) +Bu(t) \hspace{0.1cm} \text{ for } \hspace{0.1cm} t\in[0,\tau],\\
x(t)=\phi(t) \hspace{0.1cm} \text{ for } \hspace{0.1cm} t\in]-\infty,0].
\end{array}\right.
\end{equation*}

Next, we verifies that assumptions in Theorem \ref{thm 2} are satisfied.\\
$\bullet$ $\bf\huge{Assumptions:\hspace{0.1cm} (H_0),\hspace{0.1cm} (H_2),\hspace{0.1cm} (H_4), \hspace{0.1cm} (H_5)}$.  
\begin{enumerate}
	\item[] - $Assumption$ $\bf\huge(H_0)$: let $0<\alpha<\beta<1$. We have,  
	\begin{center}
		$A^{\alpha}(t)x=\sum\limits_{n=1}^{+\infty}(n^2-b(t))^{\alpha}\langle x,e_n\rangle e_n$ \text{ for }  $x\in D(A^{\alpha}(t))$,
	\end{center}
	and
	\begin{center}
		$A^{-\beta}(t)x=\sum\limits_{n=1}^{+\infty}\dfrac{1}{(n^2-b(t))^{\beta}}\langle x,e_n\rangle e_n$ \text{ for } $x\in D(A^{\alpha}(t))$.
	\end{center}
	Then,
	\begin{eqnarray*}
		\Vert A^{\alpha}(t)A^{-\beta}(s)x\Vert_X &\leq &\left( \sum\limits_{n=1}^{+\infty}\frac{(n^2-b(t))^{\alpha}}{(n^2-b(s))^{\beta}}\vert\langle x,e_n\rangle\vert^{2}\right)^{1/2}\\
		&=& \left( \sum\limits_{n=1}^{+\infty}\frac{(n^2-b(t))^{\alpha}}{(n^2-b(t))^{\beta}}\frac{(n^2-b(t))^{\beta}}{(n^2-b(s))^{\beta}} \vert\langle x,e_n\rangle\vert^{2}\right)^{1/2}\\
		&\leq & \left( \sum\limits_{n=1}^{+\infty} \left(\frac{n^2-b(t)}{n^2-b(s)}\right)^{\beta} \vert\langle x,e_n\rangle\vert^{2}\right)^{1/2}\\
		&\leq & \left( \sum\limits_{n=1}^{+\infty} \left(\frac{n^2-b(t)}{n^2+1}\right)^{\beta}\!\! \vert\langle x,e_n\rangle\vert^{2}\right)^{1/2},\hspace{0.1cm} (\text{since } b(s)<-1).
	\end{eqnarray*}
	Thus, there exists $C_{\alpha,\beta}>0$ such that
	\begin{center}
		$\Vert A^{\alpha}(t)A^{-\beta}(s)\Vert_{\mathcal{  L}(X)}\leq C_{\alpha,\beta}$ \text{ for } $0<\alpha<\beta<1$.
	\end{center}
	In what follows, we show that $A^{\frac{1}{2}}(v)R(t,s)=R(t,s)A^{\frac{1}{2}}(v)$ on $D(A^{\frac{1}{2}}(t_0))$ for $0\leq s\leq t$, and $v\in \mathbb{R}^{+}$. Let $x\in D(A^{\frac{1}{2}}(t_0))$, by Lemma \ref{lemma 3}, we have
	\begin{eqnarray*}
		R(t,s)A^{\frac{1}{2}}(v)x &=& U(t,s)A^{\frac{1}{2}}(v)x+\displaystyle\int_{s}^{t}U(t,r)Q(r,s)A^{\frac{1}{2}}(v)x dr,
	\end{eqnarray*}
	and 
	\begin{eqnarray*}
		A^{\frac{1}{2}}(v)R(t,s)x
		&=& A^{\frac{1}{2}}(v) U(t,s)x + \displaystyle\int_{s}^{t}A^{\frac{1}{2}}(v)U(t,r)Q(r,s)x dr\\
		&=& U(t,s)A^{\frac{1}{2}}(v)x + \displaystyle\int_{s}^{t}U(t,r)A^{\frac{1}{2}}(v)Q(r,s)x dr.
	\end{eqnarray*}
	Hence,
	\begin{equation}
	\begin{array}{l}
	A^{\frac{1}{2}}(v)R(t,s)x- R(t,s)A^{\frac{1}{2}}(v)x= \displaystyle\int_{s}^{t}U(t,r)\left[A^{\frac{1}{2}}(v)Q(r,s)x- Q(r,s)A^{\frac{1}{2}}(v)x\right]dr.
	\end{array}
	\label{6.2}
	\end{equation}
	On other hand, we have
	\begin{equation}
	\left.\begin{array}{l}
	A^{\frac{1}{2}}(v)Q(r,s)x\\
	=A^{\frac{1}{2}}(v)G(r,r)\displaystyle\int_{s}^{r}R(w,s)x dw-A^{\frac{1}{2}}(v)\displaystyle\int_{s}^{r}\dfrac{\partial G(r,w)}{\partial w}\displaystyle\int_{s}^{w}R(l,s)xdldw\\
	= G(r,r)\displaystyle\int_{s}^{r}A^{\frac{1}{2}}(v)R(w,s)x dw-\displaystyle\int_{s}^{r}\dfrac{\partial G(r,w)}{\partial w}\displaystyle\int_{s}^{w}A^{\frac{1}{2}}(v)R(l,s)xdldw,
	\end{array}\right.
	\label{6.3}
	\end{equation}
	and 
	\begin{equation}
	\begin{array}{l}
	Q(r,s)A^{\frac{1}{2}}(v)x 
	= G(r,r)\displaystyle\int_{s}^{r}R(w,s)A^{\frac{1}{2}}(v)x dw-\displaystyle\int_{s}^{r}\dfrac{\partial G(r,w)}{\partial w}\displaystyle\int_{s}^{w}R(l,s)A^{\frac{1}{2}}(v)xdldw.
	\end{array}
	\label{6.4}
	\end{equation}
	Combining \eqref{6.2} with \eqref{6.3} and \eqref{6.4}, we obtain the following
	\begin{equation}
	\left.
	\begin{array}{l}
	A^{\frac{1}{2}}(v)R(t,s)x- R(t,s)A^{\frac{1}{2}}(v)x\\
	= \displaystyle\int_{s}^{t}U(t,r)\left[G(r,r)\displaystyle\int_{s}^{r}\left[A^{\frac{1}{2}}(v)R(w,s)x-R(w,s)A^{\frac{1}{2}}(v)x\right]dw\right.\\
	-\left. \displaystyle\int_{s}^{r} \dfrac{\partial G(r,w)}{\partial w} \displaystyle\int_{s}^{w}\left[A^{\frac{1}{2}}(v)R(l,s)x-R(l,s)A^{\frac{1}{2}}(v)x\right]dl dw \right] dr.
	\end{array}\right.
	\label{6.5}
	\end{equation}
	Let $K(t,v,s)= A^{\frac{1}{2}}(v)R(t,s)x- R(t,s)A^{\frac{1}{2}}(v)x$, and  $H(r,v,s)=\Delta\displaystyle\int_{s}^{r}K(w,v,s)dw$. Since $G(t,s)=C(t,s)\Delta$ for $0\leq s\leq t$,
	by \eqref{6.3} and \eqref{6.4}, we obtain that
	\begin{eqnarray*}
		H(r,v,s) &=& \dfrac{1}{C(r,r)}\left[A^{\frac{1}{2}}(v)Q(r,s)x-Q(r,s)A^{\frac{1}{2}}(v)x\right] +\dfrac{1}{C(r,r)}\displaystyle\int_{s}^{r} \frac{\partial C(r,l)}{\partial l}H(l,v,s)dwdl.
	\end{eqnarray*} 
	Then, $H(r,v,s)$ is the unique solution (in $X$) of the following equation:
	\begin{eqnarray*}
		S(r,s)&=&\frac{1}{C(r,r)}\left[A^{\frac{1}{2}}(v)Q(r,s)x-Q(r,s)A^{\frac{1}{2}}(v)x\right] +\frac{1}{C(r,r)}\displaystyle\int_{s}^{r}\frac{\partial C(r,l)}{\partial l}S(l,s)dl,
	\end{eqnarray*}
	with the unknown $S(\cdot,\cdot)$. That is, $\Delta\displaystyle\int_{s}^{r}K(w,v,s)dw$ exists in $X$, i.e,
	\begin{equation}
	\displaystyle\int_{s}^{r}K(w,v,s)dw\in D(\Delta) \text{ for } r\geq s \geq 0.
	\label{6.6}
	\end{equation}
	Let $\bar{K}(t):=\bar{K}(t,v,s)=\Vert K(t,v,s)\Vert_{X} $. From \eqref{6.5}, using \eqref{6.6} and the fact that $\Delta U(t,s)\in \mathcal{L}(X)$, under the aboves conditions on $C(\cdot,\cdot)$, we can show that there exist non-negatives continuous functions $M_1(\cdot,\cdot)$ and $M_2(\cdot,\cdot,\cdot)$ such that 
	\begin{eqnarray*}
		\bar{K}(t) &\leq & \displaystyle\int_{s}^{t}\displaystyle\int_{s}^{r} M_1(t,r) \bar{K}(w)dw dr + \displaystyle\int_{s}^{t}\displaystyle\int_{s}^{r}\displaystyle\int_{s}^{w} M_2(t,r,w) \bar{K}(l) dl dw dr.
	\end{eqnarray*}
	As $x\in D(A^{\frac{1}{2}}(t_0))$, $\bar{K}(t)$ is continuous in $t$. It follows then from Fubini's Theorem that
	\begin{eqnarray*}
		\bar{K}(t) &\leq & \displaystyle\int_{s}^{t}\bar{K}(w)\displaystyle\int_{w}^{t} M_1(t,r) dr dw + \displaystyle\int_{s}^{t}\displaystyle\int_{s}^{r}\bar{K}(l)\displaystyle\int_{l}^{r} M_2(t,r,w) dw dl dr\\
		&\leq & \displaystyle\int_{s}^{t}\bar{K}(w)\displaystyle\int_{w}^{t} M_1(t,r) dr dw + \displaystyle\int_{s}^{t}\bar{K}(l)\displaystyle\int_{l}^{t}\displaystyle\int_{l}^{r} M_2(t,r,w) dw dr dl.
	\end{eqnarray*}
	Let $a>0$, then
	\begin{eqnarray*}
		\bar{K}(t) &\leq & \displaystyle\int_{s}^{t}\left[\bar{M}_{1,a}+ \bar{M}_{2,a}\right] \bar{K}(w) dw \hspace{0.1cm} \text{ for } \hspace{0.1cm} 0\leq s \leq t\leq a,
	\end{eqnarray*}
	where 
	\begin{center}
		$\bar{M}_{1,a}=\sup\limits_{0 \leq w\leq t\leq a}\displaystyle\int_{w}^{t} M_1(t,r) dr$ \text{ and } $\bar{M}_{2,a}=\sup\limits_{0 \leq w\leq t\leq a}\displaystyle\int_{w}^{t}\displaystyle\int_{w}^{r} M_2(t,r,l) dl dr$,
	\end{center}
	for $a> 0$. By Gr\"onwall Lemma, we get that $\bar{K}(t)=0$ for $t\in[0,a]$. Since $a>0$ is arbitrary, we obtain that $\bar{K}(t)=0$ for each $t\geq 0$,  i.e,  $A^{\frac{1}{2}}(v)R(t,s)= R(t,s)A^{\frac{1}{2}}(v)$ on $D(A^{\frac{1}{2}}(t_0))$ for $0\leq s\leq t$, and $v\in \mathbb{R}^{+}$. Thus, assumption $\bf\huge(H_0)$ holds.
	\item[] - $Assumption$ $\bf\huge(H_2)$: 
	the maps $R(\lambda,-A(t))$, $t\geq 0$ are compact for each $\lambda>0$. By \cite[Proposition 2.1]{36}, we obtain that $U(t,s)$ is compact whenever $t-s>0$. Then, $\bf\huge(H_2)$ holds.
	\item[] - $Assumption$ $\bf\huge(H_4)$: let $x:]-\infty,\tau]$, $x_0=\phi$ and $x_{\mid_{[0,\tau]}}\in \mathcal{  PC}_{g,\frac{1}{2}}$. By continuity of $\rho$, $h$ and $t\to x_t$ we can see that function $t\to F(t,x_{\rho(t,x_t)})$ is strongly measurable on $[0,\tau]$ and $t\to F(s,x_t)$ is continuous on $\Lambda(\rho^{-})\cup[0,\tau]$ for every $s\in[0,\tau]$ and $F(t,\cdot):\mathcal{  H}_{\frac{1}{2}}\to X$ is continuous for a.e $t\in[0,\tau]$. Moreover, $\Vert F(t,\psi)\Vert_X\leq M_h$ for each $(t,\psi)\in [0,\tau]\times\mathcal{  H}_{\frac{1}{2}}$. Consequently, assumption $\bf\huge(H_4)$ is satisfied.
	
	\item[] - $Assumption$ $\bf\huge(H_5)$: we have $B^{*}=B$, where $B^*$ is the adjoint operator of $B$. In fact, let $u,v\in X$, then 
	\begin{eqnarray*}
		\langle Bu,v\rangle &=&\displaystyle\int_{0}^{\pi}\chi_{]a_1,a_2[}(\xi)u(\xi)v(\xi)d\xi\\
		&=&\displaystyle\int_{a_1}^{a_2}u(\xi)v(\xi)d\xi=\displaystyle\int_{0}^{\pi}u(\xi)\chi_{]a_1,a_2[}(\xi)v(\xi)d\xi\\
		&=&\langle u,Bv\rangle.
	\end{eqnarray*}
	Since $A(t)^{*}=A(t)$ and $G(t,s)^{*}=G(t,s)$ for $t\geq s\geq 0$ where $A(t)^{*}$ and $G(t,s)^{*}$ are the adjoint operators of $A(t)$ and $G(t,s)$ respectively, with the same argument as in the proof of Theorem 3.10 from \cite{Grimmer 1984}, we show that $R(t,s)^{*}=R(t,s)$ for $t\geq s\geq 0$. Therefore, $B^{*}R(t,s)^{*}x^*=0\Rightarrow (R(t,s)x^*)(\xi)=0$ for $\xi\in]a_1,a_2[$ and $0\leq s\leq t\leq \tau$. By Lemma \ref{lemma 3}, we have 
	\begin{center}
		$R(t,s)x^{*}=U(t,s)x^{*}+\displaystyle\int_{s}^{t}U(t,r)Q(r,s)x^{*} dr$ \text{ for } $t\geq s\geq 0$.
	\end{center}
	Since $(R(t,s)x^{*})(\xi)=0$ for $\xi\in ]a_1,a_2[$ and $0\leq s\leq t\leq \tau$, it follows that $Q(r,s)x^{*}=0$ on $]a_1,a_2[$ for $0\leq s\leq r\leq t$. Then, we obtain that $(U(t,s)x^{*})(\xi)=(R(t,s)x^{*})(\xi)=0$ for $\xi\in ]a_1,a_2[$ and $0\leq s\leq t\leq\tau$. That is,
	\begin{center}
		$\sum\limits_{n=1}^{+\infty}\exp\left({-n^{2}(t-s)+\int_{s}^{t}b(r)dr}\right)\langle x^{*},e_n\rangle e_n(\xi)=0$,
	\end{center}
	for  $\xi\in]a_1,a_2[$ and  $0\leq s\leq t\leq\tau$.
	
	Since the function $\xi\to \sum\limits_{n=1}^{+\infty}\exp\left({-n^{2}(t-s)+\int_{s}^{t}b(r)dr}\right)\langle x^{*},e_n\rangle e_n(\xi)$ is analytic, it follows that
	\begin{center}
		$\sum\limits_{n=1}^{+\infty}\exp\left({-n^{2}(t-s)+\int_{s}^{t}b(r)dr}\right)\langle x^{*},e_n\rangle e_n(\xi)=0$,
	\end{center}
	for $\xi\in(0,\pi)$, and  $0\leq s\leq t\leq \tau$. Hence,  $\langle x^*,e_n\rangle =0$ for $n\geq 1$, that is $x^{*}=0$. From Theorem \ref{thm 1}, we show that $\langle\mathcal{ B}_\tau x^*,x^*\rangle>0$ for $x^*\in X^*-\{0\}$. Thus, the linear part corresponding to equation (\ref{17}) is approximately controllable on $[0,\tau]$. As a consequence, assumption $\bf\huge(H_5)$ is satisfied.
\end{enumerate}

Consequently, by applying Theorem \ref{thm 2}, we get the following result.
\begin{prop}
	\itshape Equation $(\ref{17})$ is approximately controllable on $[0,\tau]$.
\end{prop}
     \section*{Author Contributions} 
     \noindent
     
     Writing - Original Draft Preparation: Mohammed ELGHANDOURI.
     
      Methodology and Supervision: Khalil Ezzinbi.
      
      Review and Editing: Mamadou Abdoul DIOP.
     
     \section*{Acknowledgments}
     $^*$The third  author  was partially supported by a grant from the Simons Foundation (SF).
     
     \section*{Financial disclosure}
     
     The authors declare that they have no known competing financial interests or personal relationships that could have appeared to influence the work reported in this paper.
     
     \section*{Competing Interests and Funding}
     This work does not have any conflicts of interest.
     
     \end{document}